\documentclass[12pt]{amsart}
\usepackage{amsmath}
\usepackage{amsfonts}
\usepackage{amsthm}
\usepackage{enumerate}
\usepackage{graphicx}
\usepackage{overpic}
\usepackage{a4wide}
\usepackage{amssymb} 
\usepackage{pictex}
\usepackage{rotating}  
\usepackage{xcolor}
\usepackage{etex}
\usepackage{cite} 
\usepackage{float}
\usepackage{mathtools}
\usepackage[utf8]{inputenc}
\usepackage{comment}

\numberwithin{equation}{section}

\newtheorem{theorem}{Theorem}[section]

\newtheorem*{claim}{Claim}

\newtheorem{proposition}[theorem]{Proposition}
\newtheorem{lemma}[theorem]{Lemma}
\newtheorem{corollary}[theorem]{Corollary}
\newtheorem{Definition}[theorem]{Definition}
\newenvironment{definition}{\begin{Definition}\rm}{\end{Definition}}
\newtheorem{Remark}[theorem]{Remark}
\newenvironment{remark}{\begin{Remark}\rm}{\end{Remark}}
\newtheorem{RHproblem}[theorem]{RH problem}

\newtheorem{Example}[theorem]{Example}

\theoremstyle{definition}
\newtheorem{question}{Question}

\theoremstyle{plain}
\newtheorem*{theoA}{Theorem A}
\newtheorem*{theoB}{Theorem B}

\newcommand{\C}{\mathbb{C}}

\newcommand{\Z}{\mathbb{Z}}
\newcommand{\N}{\mathbb{N}}
\newcommand{\R}{\mathbb{R}}
\newcommand{\Q}{\mathbb{Q}}
\newcommand{\T}{\mathbb{T}}

\newcommand{\LL}{\mathcal L}

\newcommand{\orb}{\text{Orb}}
\newcommand{\aorb}{\overline{\text{Orb}}}

\newcommand{\Span}{{\rm Span} \,}
\newcommand{\Codim}{{\rm Codim} \,}

\renewcommand{\tilde}{\widetilde}

\usepackage[bookmarksopen, naturalnames]{hyperref}

\makeatother

\begin{document}

\title[Hypercyclic subsets]{Hypercyclic subsets}
\author{S. Charpentier, R. Ernst}
\address{St\'ephane Charpentier, 
Institut de Math\'ematiques, UMR 7373, Aix-Marseille Universit\'e, 39 rue F. Joliot Curie, 13453 Marseille Cedex 13, FRANCE}
\email{stephane.charpentier.1@univ-amu.fr}
\address{Romuald Ernst,
LMPA, Centre Universitaire de la Mi-Voix, Maison de la Recherche Blaise-Pascal, 50 rue Ferdinand Buisson, BP 699, 62228 Calais Cedex}
\email{romuald.ernst@math.cnrs.fr}

\begin{abstract}
	We completely characterize the finite dimensional subsets $C$ of any separable Hilbert space for which the notion of $C$-hypercyclicity coincides with the notion of hypercyclicity, where an operator $T$ on a topological vector space $X$ is said to be $C$-hypercyclic if the set $\{T^nx,\,n\geq 0,\,x\in C\}$ is dense in $X$. We give a partial description for non necessarily finite dimensional subsets. We also characterize the finite dimensional subsets $C$ of any separable Hilbert space $H$ for which the somewhere density in $H$ of $\{T^nx,\,n\geq 0,\,x\in C\}$ implies the hypercyclicity of $T$. We provide a partial description for infinite dimensional subsets. These improve results of Costakis and Peris, Bourdon and Feldman, and Charpentier, Ernst and Menet.
\end{abstract}

\maketitle

\section{Introduction}

At the core of Linear Dynamics is the notion of hypercyclicity. A continuous linear operator $T$ from a topological vector space $X$ into itself is said to be \emph{hypercyclic} if there exists a vector $x$ in $X$ whose orbit $\orb(x,T):=\{T^nx,\,n\geq 0\}$ is dense in $X$.
Such an $x$ is called a hypercyclic vector for $T$. In the following we will assume that $X$ is a \emph{complex} topological vector space.
The translation operator $T_a:f\mapsto f(\cdot -a)$, $a\neq 0$, acting on the Fr\'echet space of entire functions $H(\C)$ is known after Birkhoff as the first example of hypercyclic operator \cite{birkhoff_demonstration_1929}.
Later, natural operators such as the differentiation operator or the dilation of the backward shift were shown by MacLane \cite{maclane_sequences_1952} and Rolewicz \cite{rolewicz_orbits_1969} to be hypercyclic on $H(\C)$ and the space of square summable sequences $\ell^2(\N)$ respectively.
The systematic study of the abstract notion of hypercyclicity became quite active since the early eighties, after Kitai stated a useful criterion for hypercyclicity \cite{kitai_invariant_1982}.
What is now referred to as the Hypercyclicity Criterion, a refinement by Bès \cite{bes_three_1998} of the initial Kitai's one, has been proven by De La Rosa and Read \cite{de_la_rosa_hypercyclic_2009} to be not satisfied by every hypercyclic operator.
Later, Bayart and Matheron \cite{bayart_dynamics_2009} refined De la Rosa and Read's counterexample and exhibited one on any classical Banach spaces, including the Hilbert space.
This result answered a rather long standing open problem posed by Herrero in 1993, which led at the time to some natural questions about the definition of hypercyclicity.
We may state two of them: Does the density of the union of finitely many orbits imply the density of one orbit? Does the somewhere density of one orbit imply its density everywhere? We refer to the very nice books \cite{bayart_dynamics_2009} and \cite{grosse-erdmann_linear_2011} for a quite rich insight about linear dynamics and for (much) more details on these questions.

The rigidity imposed by the linearity of the operators makes these kinds of questions relevant. Indeed Costakis and Peris independently gave a positive answer to the first one \cite{costakis_conjecture_2000,peris_multi-hypercyclic_2001}.
Later on, Le\'on and M\"uller \cite{leon-saavedra_rotations_2004} proved that the density of the orbit
$$\orb(\T x,T):=\left\{\lambda T^nx,\,n\geq 0,\,|\lambda|=1\right\}$$
of the one dimensional uncountable set $\T x:=\{\lambda x,\, |\lambda|=1\}$ under $T$ automatically implies the density of $\orb(x,T)$ itself; a corollary of that result is that $T$ is hypercyclic if and only if $\lambda T$ is for every complex number $\lambda$ of modulus $1$. The proof of Costakis-Peris' result contained the foundations of the stronger and beautiful result by Bourdon and Feldman \cite{bourdon_somewhere_2003} asserting that only the somewhere density of the orbit $\orb(x,T)$ is needed to ensure the hypercyclicity of $x$ for $T$. The proof of Le\'on-M\"uller's result exploited the group structure of the unit circle $\T$, an idea which was then developed and extended to statements in terms of groups and semigroups \cite{conejero_hypercyclic_2007,shkarin_universal_2008,matheron_subsemigroups_2012}, see also \cite[Chapter 3]{bayart_dynamics_2009}.

Recently, Le\'on-M\"uller and Bourdon-Feldman's results have been improved by a complete description of the subsets $\Gamma$ of $\C$ satisfying one of the following two properties \cite{charpentier_-supercyclicity_2016}:\\
\indent{}\textbf{(P)} for every complex Banach space $X$, for every operator $T$ on $X$ and for any $x\in X$,
$$\orb(\Gamma x,T):=\{\gamma T^nx,\,n\geq 0,\,\gamma \in \Gamma\}\text{ is dense in }X\text{ iff }\orb(x,T)\text{ is dense in }X;$$\\
\indent{}\textbf{(P')} for every complex Banach space $X$, for every operator $T$ on $X$ and for any $x\in X$,
$$\orb(\Gamma x,T)\text{ is somewhere dense in }X\text{ iff }\orb(x,T)\text{ is dense in }X.$$
The sets $\Gamma$ satisfying Property \textbf{(P)} turn out to be exactly those which are bounded and bounded away from $0$ (after removing the single point $0$) \cite[Theorem A]{charpentier_-supercyclicity_2016}, and the sets $\Gamma$ satisfying Property \textbf{(P')} are those satisfying \textbf{(P)} and such that the set $\Gamma \T :=\{\gamma \lambda,\,\gamma \in \Gamma,\,|\lambda|=1\}$ is nowhere dense in $\C$ \cite[Theorem B]{charpentier_-supercyclicity_2016}. The notion of $\Gamma$-supercyclicity was introduced: given $\Gamma \subset \C$, an operator $T$ on $X$ is said to be $\Gamma$-supercyclic if there exists $x$ in $X$ such that the orbit $\orb(\Gamma x,T)$ is dense in $X$. $\Gamma$-supercyclicity extends the notions of hypercyclicity and \emph{supercyclicity} - introduced by Hilden and Wallen \cite{hilden_cyclic_1973} - which corresponds to $\Gamma =\C$. Thus \cite[Theorem A]{charpentier_-supercyclicity_2016} says that for some \emph{good} non-empty subsets $C$ of a one dimensional subspace of $X$, possibly open in this subspace, the density in $X$ of the set $\{T^nx,\,n\geq 0,\,x\in C\}$ implies the hypercyclicity of $T$. However neither this result nor Costakis and Peris' one implies the other and we may expect a positive statement covering both multihypercyclicity and $\Gamma$-supercyclicity. The following natural question arises:

\begin{question}\label{Quest1}
Let $x_1,\ldots,x_N$ be a family of pairwise distinct vectors in $X$ and $\Gamma_1,\ldots,\Gamma_N$ be subsets of $\C^*$ which are bounded and bounded away from zero. If $\orb\left(\cup_{i=1}^{N}\Gamma_i x_i,T\right)=\cup_{i=1}^{N}\orb\left(\Gamma_i x_i,T\right)$ is dense in $X$, is some $x_i$ hypercyclic for $T$?
\end{question}
The classical Costakis and Peris' result is a consequence of the Bourdon-Feldman Theorem. According to \cite[Theorem B]{charpentier_-supercyclicity_2016}, subsets of $\C$ of the form $[a,b]\T$ - bounded and bounded away from $0$, but somewhere dense in $\C$ whenever $a<b$ - do not satisfy the property \textbf{(P')}. This result indicates that a positive answer to Question \ref{Quest1} is not likely to be obtained, as for the classical Costakis and Peris' result, as an application of some Bourdon-Feldman type Theorem.
 
Question \ref{Quest1} eventually gives rise to an extended notion of $\Gamma$-supercyclicity for multi-dimensional $\Gamma\subset\C^l$, or even for subsets of spaces of infinite sequences. Without going into details now, we are naturally led to introduce what we call \emph{hypercyclic scalar subsets} (see Paragraph \ref{hyp-sca-sets}) and \emph{Bourdon-Feldman scalar subsets} (see Paragraph \ref{BF-sca-sets}). Roughly speaking, they are those multi-dimensional scalar subsets which satisfy a property similar to \textbf{(P)} and \textbf{(P')} respectively. The following question was already mentioned in \cite{charpentier_-supercyclicity_2016}.
\begin{question}\label{Quest2}
Does there exist a characterization of multi-dimensional hypercyclic scalar subsets and Bourdon-Feldman scalar subsets?
\end{question}

Now, let us briefly come back to the complete description of one dimensional hypercyclic scalar subsets (\emph{i.e.} of subsets of $\C$ satisfying Property \textbf{(P)}) given in \cite{charpentier_-supercyclicity_2016}. It allows for instance to assert that given any fixed separable Banach space $X$ and any fixed non-zero $x\in X$, the density of $\orb\left([1,2]x,T\right)$ in $X$ automatically implies that of $\orb\left(x,T\right)$; but it does not say whether $\aorb\left([0,1]x,T\right)=X$ does \emph{not} imply $\aorb\left(x,T\right)=X$. It only says that there \emph{exist some} $X$, \emph{some} $T\in \LL(X)$ and \emph{some} $x\in X$ such that the previous implication does not hold. This is actually rather unsatisfying as one may prefer an answer to the much more precise questions: Given $X$ and $C\subset X$,
\begin{itemize}\item does the density of $\{T^nx,\,n\geq 0,\,x\in C\}$ in $X$ imply that $T$ is hyperyclic?
\item on the contrary, does there exist a \emph{non-hypercyclic} $T\in \LL(X)$ such that $\{T^nx,\,n\geq 0,\,x\in C\}$ is dense in $X$?
\end{itemize}
Similar questions related to the Bourdon-Feldman Theorem make also sense. Thus the notion of $\Gamma$-supercyclicity is, at least from this point of view, a bit soft and one may want to replace it with a notion which depends on the ambient space.
\begin{definition}\label{def-C-hyp}Let $X$ be a Banach space and $C$ be a subset of $X$. We say that an operator $T$ on $X$ is $C$-hypercyclic 
 if
$$\orb(C,T):=\{T^nx,\,n\geq 0,\,x\in C\}$$
is dense in $X$.
\end{definition}
If $C$ is  a single point, we recover the classical notion of hypercyclicity; if $C$ is a finite union of points, that of multihypercyclicity considered in \cite{costakis_conjecture_2000,peris_multi-hypercyclic_2001}. Given a fixed $x\in X$ and $\Gamma\subset\C$, any $\Gamma x$-hypercyclic operator is in particular $\Gamma$-supercyclic (but a $\Gamma$-supercyclic operator is not necessarily $\Gamma x$-hypercyclic...). Regarding to the previous, the main notions of this paper are the following.
\begin{definition}\label{def-hyp-set}Let $C$ be a subset of a separable Banach space $X$.
\begin{enumerate}\item We say that $C$ is a \emph{hypercyclic subset}\footnote{Not to be confused with the notion of \emph{hypercyclicity set} as introduced in \cite{Bes2016}} if $C\setminus\{0\}$ is non-empty and any $C$-hypercyclic operator on $X$ is hypercyclic.
\item We say that $C$ is a \emph{Bourdon-Feldman subset} if $C\setminus\{0\}$ is non-empty and it satisfies the following property: for any operator $T\in \LL(X)$, the somewhere density in $X$ of the set $\orb(C,T)$ implies that $T$ is hypercyclic.
\end{enumerate}
\end{definition}
The following question - partially posed in \cite{charpentier_-supercyclicity_2016} - completes Questions \ref{Quest1} and \ref{Quest2}.
\begin{question}\label{Quest3}
Does there exist a characterization of hypercyclic and Bourdon-Feldman subsets?  
\end{question}
This direction of research was pursued by several authors in different contexts. Most of them, like Costakis and Peris \cite{costakis_conjecture_2000,peris_multi-hypercyclic_2001}, Le\'on-M\"uller \cite{leon-saavedra_rotations_2004} or Charpentier-Ernst-Menet \cite[Theorem A]{charpentier_-supercyclicity_2016}, give examples of hypercyclic subsets. Also, Feldman \cite{feldman_countably_2003} proved that in the case where $T$ is a weighted backward shift on $\ell^2(\N)$, if $T$ has a bounded set with dense orbit then $T$ is hypercyclic. His result asserts that any bounded set is a hypercyclic subset if we restrict ourselves to weighted shifts on $\ell^2(\N)$. A few papers give non-trivial examples of sets which are not hypercyclic subsets. Among them, it is worth mentioning \cite{feldman_countably_2003} which is interested in the notion of \emph{countably hypercyclic operators}, namely in those $T$ such that $\orb(\{x_n\}_n,T)$ is dense in $X$ for some (infinite) bounded \emph{separated sequence} $(x_n)_n$, where separated means that there exists $\delta >0$ such that $\Vert x_n-x_m\Vert \geq \delta$ for any $n\neq m$. Remark that a countably hypercyclic operator for a separated sequence $(x_n)_{n\in\N}$ is a $C$-hypercyclic operator with $C=\{x_n\}_{n\in\N}$. In \cite{feldman_countably_2003}, the following question was posed.
\begin{question}\label{Quest4}
	Does there exist a countably hypercyclic operator being not hypercyclic?
\end{question}
We recall that by Costakis and Peris' result the answer is no if one only considers finite sequences. At the end of \cite{feldman_countably_2003}, the author mentions that a positive answer to his question was given by Peris in a private communication. For an explicit example solving this question by the positive, see \cite[Exercise 6.3.3]{grosse-erdmann_linear_2011}. Here again, the answer may look a bit disappointing, as it consists in exhibiting a specific Hilbert space, a specific bounded separated sequence $(x_n)_n$, and a non-hypercyclic $T\in \LL(X)$ such that $\orb(\{x_n\}_n,T)$ is dense in $X$. What about \emph{any} fixed bounded separated sequences in \emph{any} Banach spaces and the following very general question?
\begin{question}\label{Quest5}Given a bounded separated sequence $(x_n)_{n\in\N}$ in a separable Banach space $X$, does there \emph{always} exist a countably hypercyclic operator $T$ for $(x_n)_{n\in\N}$ which is not hypercyclic? 
\end{question}

\medskip{}
The purpose of this paper is to attack Questions \ref{Quest1} to \ref{Quest5}. We will obtain a complete answer to Question \ref{Quest1} (Theorem \ref{thm2}) and Question \ref{Quest2} (Theorems \ref{thm-main} and \ref{thm-main-l2}). The proof of Theorem \ref{thm2} will consist in an adaptation of that given by Peris for multihypercyclicity \cite{peris_multi-hypercyclic_2001}. The ``only if'' parts in Theorems \ref{thm-main} and \ref{thm-main-l2}) are partially based on geometric arguments in Hilbert spaces. These results will surprisingly turn out to be very useful in order to obtain, thanks to the fact that (quasi-)conjugacy preserves the dynamical properties of operators, an \emph{almost complete} answer to Question \ref{Quest3} (Theorems~A and B below) when $X$ is a separable Hilbert space. As a corollary of Theorem~A, we will get a positive answer to Question \ref{Quest5} in the Hilbert setting. In particular, it will tell us that the most natural bounded separated sequence in a Hilbert space - namely an orthonormal basis - is not a hypercyclic subset (see Question \ref{Quest5}). In order to state Theorems~A and B, we introduce the following terminology, which will repeatedly appear throughout the paper.

\begin{definition}\label{vector-annulus}We say that a subset $C$ of $X$ is a \emph{vector annulus} if there exist $x$ in $X$ and $0<a\leq b<\infty$ such that
$$C=[a,b]\T x:=\left\{re^{\imath\theta}x,\,a\leq r\leq b, 0\leq\theta\leq 2\pi\right\}.$$
\end{definition}

\begin{theoA}Let $C$ be a subset of a separable Hilbert space $H$.
\begin{enumerate}
\item We assume that $C$ is contained in a finite dimensional subspace of $H$. Then $C$ is a hypercyclic subset if and only if $C\setminus \{0\}$ is non-empty and contained in a finite union of vector annuli.
\item If $C$ is not contained in a finite dimensional subspace of $H$ and if it contains a sequence $(x_n)_n$ of linearly independent vectors satisfying
\begin{equation}\label{eq-codim-main-thm}
\text{Codim}(\overline{\text{Span}}(x_n,\, n\geq 0))=\infty,
\end{equation}
then $C$ is not a hypercyclic subset.
\end{enumerate}
\end{theoA}

The first part of Theorem~A gives a complete characterization of hypercyclic subsets among finite dimensional subsets of a separable Hilbert space. It also answers Question 6 from \cite{charpentier_-supercyclicity_2016} and provides with a wide class of examples of sets which are not hypercyclic. For example, in the Hilbert setting, a segment joining two linearly independent points, a non-trivial sphere, or a non-empty open set of $X$ is never a hypercyclic subset. The second part of Theorem~A tends to suggest that a hypercyclic subset is necessarily contained in a finite dimensional subspace. Actually Theorem~A, Part 2, can equivalently be stated in terms of \emph{almost overcomplete sequences} (see Section \ref{pf-thmA} for details). Thus Theorem~A does only let open the following question: Are there almost overcomplete sequences in a separable Hilbert space which are hypercyclic subsets? We refer the reader to the last section devoted to open questions. The proof of the second part of Theorem~A, like that of the first one, is based on a construction made in the proof of Theorem \ref{thm-main-l2}. An obstruction occurs when we deal with linearly independent sequences spanning closed subspaces with finite codimension.

Finally, the previous considerations combined with \cite[Theorem B]{charpentier_-supercyclicity_2016} will allow us to obtain an almost complete description of Bourdon-Feldman scalar subsets and, at end, to obtain the following.
\begin{theoB}Let $C$ be a subset of a separable Hilbert space $H$.
\begin{enumerate}
	\item If $C$ is finite dimensional, $C$ is a Bourdon-Feldman subset if and only if $C\setminus\{0\}$ is non-empty and there exist $x_1,\ldots,x_N$ in $X$ and $\Gamma _1,\ldots,\Gamma _N$ subsets of $\C$, with $\Gamma _i\setminus \{0\}$ bounded and bounded away from $0$ and $\Gamma _i\T$ nowhere dense in $\C$ for every $i\in\{1,\ldots,N\}$, such that
	\[
	C\subset \bigcup _{i=1}^N\Gamma _ix_i.
	\]
	\item If $C$ is infinite dimensional and if $C$ contains a sequence $(x_n)_n$ of linearly independent elements satisfying \[\Codim(\overline{\Span}(x_n,\,n\geq 0))=\infty,\]
	then $C$ is not a Bourdon-Feldman subset.
\end{enumerate}
\end{theoB}

We mention that even if our results are stated in complex Banach or Hilbert spaces, they hold as well for real spaces, up to adequate changes. Also those of our results which are given for separable Banach spaces hold for separable Fr\'echet spaces as well. This is in particular the case for Theorem \ref{thm2}.

\medskip{}
The paper is organized as follows. The first section consists in the positive answer to Question \ref{Quest1} (Theorem \ref{thm2}). The second section is devoted to the multi-dimensional notion of $\Gamma$-supercyclicity, that is to Question \ref{Quest2} and the description of hypercyclic scalar subsets and Bourdon-Feldman scalar subsets (Theorems \ref{thm-main}, \ref{thm-main-l2} and \ref{thm-BF-scalar}). The third section deals with Questions \ref{Quest3}, \ref{Quest4} and \ref{Quest5}, which are partially or completely answered in Theorems~A and B. The last section contains some open questions.

\section{A sufficient condition for a set in $X$ to be a hypercyclic subset}

For the notion of hypercyclic subset, we refer to Definition \ref{def-hyp-set}. The following result gives a sufficient condition for a set in a given Banach space $X$ to be a hypercyclic subset.

\begin{theorem}\label{thm2}Let $(x_1,\ldots,x_N)$ be a finite family of vectors in $X$ and let $b\geq 1$. If the set
$$\text{Orb}(\bigcup _{i=1}^N[1,b]\T x_i,T)=\bigcup _{i=1}^N\text{Orb}([1,b]\T x_i,T)$$
is dense in $X$, then some $x_i$ is hypercyclic for $T$.
\end{theorem}

We recall that Theorem \ref{thm2} with $b=1$ is a consequence of \cite[Theorem 3.11]{bayart_dynamics_2009}, a semigroup version of the Bourdon-Feldman Theorem. Yet, it is worth saying that we cannot apply a Bourdon-Feldman type Theorem to prove the whole Theorem \ref{thm2}. Indeed, by \cite{charpentier_-supercyclicity_2016}, the orbit of $[1,b]\T x$ under some non-hypercyclic operator $T$, on some Banach space $X$, and for some $x\in X$, may be somewhere dense in $X$ but not everywhere dense (see also Theorem~B). Instead we will generalize the original proof of Costakis-Peris' result, as it is given in \cite{peris_multi-hypercyclic_2001}. For this purpose, we need the following two classical lemmas.
\begin{lemma}\label{lem1}
We keep the previous notations. If the orbit of $\bigcup _{i=1}^N[1,b]\T x_i$ under $T$ is dense in $X$, then the point spectrum $\sigma_p(T^*)$ of the adjoint $T^*$ of $T$ is empty. In particular, for every polynomial $p\neq 0$, $p(T)$ has dense range.
\end{lemma}
\begin{proof}[Proof of Lemma \ref{lem1}]The proof of this lemma is an easy combination of that of \cite[Lemma 1]{peris_multi-hypercyclic_2001} and \cite[Lemma 3.5]{charpentier_-supercyclicity_2016}.
\end{proof}
\begin{lemma}\label{lem2}Let $b\geq 1$ and $x,y\in X$. Either the interior of the closure of the orbits of $[1,b]\T x$ and of $[1,b]\T y$ under some operator $T$ on $X$ do not intersect, or we have
\begin{multline*}\text{int}\left(\overline{\text{Orb}}([1,b]\T x,T)\right)\subset \text{int}\left(\overline{\text{Orb}}([\frac{1}{b},b^2]\T y,T)\right)\\
\text{and }\text{int}\left(\overline{\text{Orb}}([1,b]\T y,T)\right)\subset \text{int}\left(\overline{\text{Orb}}([\frac{1}{b},b^2]\T x,T)\right).
\end{multline*}
\end{lemma}
\begin{proof}[Proof of Lemma \ref{lem2}]Observe first that if $\Gamma \subset \C$ is compact, then
$$\overline{\text{Orb}}(\Gamma x,T)=\Gamma\overline{\text{Orb}}(x,T).$$
Here and after, we will use the notation
$$I_{[1,b]}(x,T)=\text{int}\left(\overline{\text{Orb}}([1,b]\T x,T)\right).$$
Let assume that $I_{[1,b]}(x,T)$ and $I_{[1,b]}(y,T)$ do intersect. Because they are open, it implies for example that there exist $\gamma \in [1,b]\T$ and $n_1 \in \N$ such that
$$\gamma T^{n_1}x \in \overline{\text{Orb}}([1,b]\T y,T).$$
Multiplying each side of the previous by $[1,b]\T /\gamma$, we get
$$[1,b]\T\{T^{n_1}x\} \subset \frac{1}{\gamma}\overline{\text{Orb}}([1,b^2]\T y,T)\subset \overline{\text{Orb}}([\frac{1}{b},b^2]\T y,T),$$
hence
$$I_{[1,b]}(x,T) = \text{int}\left(\overline{\text{Orb}([1,b]\T x,T)\setminus [1,b]\T\{T^nx,\,n=0,\ldots,n_1-1\}}\right)\subset I_{[\frac{1}{b},b^2]}(y,T),$$
where the first equality holds because $[1,b]\T\{T^nx,\,n=0,\ldots,n_1-1\}$ has empty interior in $X$. The other inclusion is obtained in the same way.
\end{proof}

We are now ready to prove Theorem \ref{thm2}.
\begin{proof}[Proof of Theorem \ref{thm2}]
First observe that up to taking $b$ even bigger, we may suppose that the vectors $x_1,\ldots,x_N$ are pairwise independent.	
Let us first assume that $N$ is \emph{minimal} in the sense that
\[X=\bigcup _{i=1}^N\overline{\text{Orb}}([1,b]\T x_i,T)\quad \text{but}\quad \bigcup _{i=1}^{N-1}\overline{\text{Orb}}([1,c]\T y_i,T)\neq X,\]
for any $y_i \in X$, $i=1,\ldots,N-1$, and any $c\geq 1$. Then, we are going to prove that $N=1$.
Suppose that this has already been proven for a while. Then, if $N$ is not minimal, there exists $1\leq M<N$, $a\geq b\geq1$ and $z_1,\ldots,z_M$ linearly independent such that:
\[X=\bigcup _{i=1}^M\overline{\text{Orb}}([1,a]\T z_i,T)\quad \text{but}\quad \bigcup _{i=1}^{M-1}\overline{\text{Orb}}([1,c]\T y_i,T)\neq X,\]
for any $y_i \in X$, $i=1,\ldots,N-1$, and any $c\geq 1$.
Thus, by assumption, $M$ has to be equal to one. Then, by \cite[Theorem A]{charpentier_-supercyclicity_2016}, $z_1$ is hypercyclic for $T$, and thus $I_{[1,a]}(z_1,T)\cap I_{[1,a]}(x_i,T)\neq \emptyset$ for some $1\leq i\leq N$. Then, by Lemma \ref{lem2}, $I_{[1,a]}(z_1,T)\subset I_{[1/a,a^2]}(x_i,T)$. By \cite[Theorem A]{charpentier_-supercyclicity_2016} again, it follows that $x_i$ is hypercyclic for $T$ and the proof is done.

Thus, from now on, we assume that $N$ is minimal and, by contradiction, that $N> 1$. We claim that this implies the following two assertions:
\begin{enumerate}
\item $I_{[1,b]}(x_i,T)\neq \emptyset$ for every $i=1,\ldots ,N$;
\item For every $c,c'\geq 1$ and any $i\neq j$,
\begin{equation}\label{eqthm2}I_{[1/c,c]}(x_i,T)\cap I_{[1/c',c']}(x_j,T) =\emptyset.
\end{equation}
\end{enumerate}
Indeed, first, if $I_{[1,b]}(x_1,T)=\emptyset$ for example, then
$$\bigcup _{i=1}^N\overline{\text{Orb}}([1,b]\T x_i,T)=\bigcup _{i=2}^N\overline{\text{Orb}}([1,b]\T x_i,T)$$
which contradicts the minimality of $N$ and gives (1). Second, assume for instance that $I_{[1/c,c]}(x_1,T)$ intersects $I_{[1/c',c']}(x_2,T)$ for some $c,c'\geq 1$. Without loss of generality, we may suppose that $c'\geq b$. If $c'\geq c$, then $I_{[1/c',c']}(x_1,T)$ and $I_{[1/c',c']}(x_2,T)$ do intersect too, and so do $I_{[1,c'^2]}(x_1,T)$ and $I_{[1,c'^2]}(x_2,T)$; by Lemma \ref{lem2}, $I_{[1,c']}(x_1,T)\subset I_{[1,c'^2]}(x_1,T)\subset I_{[1/c'^2,c'^4]}(x_2,T)$ hence, since $c'\geq b$,
$$X=\bigcup _{i=1}^N\overline{\text{Orb}}([1,b]\T x_i,T)\subset \bigcup _{i=3}^N\overline{\text{Orb}}([1,c']\T x_i,T)\cup \overline{\text{Orb}}([1/c'^2,c'^4]\T x_2,T).$$
Up to a dilation by some positive number, we get another contradiction to the minimality of $N$. If $c'<c$, we just interchange the roles of $x_1$ and $x_2$.

\medskip{}
Let us now observe that if $x\in X$ is such that $I_{[1,b]}(x,T)\neq \emptyset$, then by Lemma \ref{lem2}, there exists some $i\in \{1,\ldots,N\}$ such that
\begin{equation}\label{eq2thm2}I_{[1,b]}(x_i,T)\subset I_{[\frac{1}{b},b^2]}(x,T).
\end{equation}
We infer that the latter inclusion implies the following one:
\begin{equation}\label{eq3}
\text{Orb}([1,b]\T x,T)\subset I_{[\frac{1}{b},b^2]}(x,T).
\end{equation}
Otherwise, by \eqref{eq2thm2}, $\text{Orb}([1,b]\T x,T)\not\subset I_{[1,b]}(x_i,T)$ and there must exist $\gamma \in [1,b]$ and $n \in \N$ such that $\gamma T^{n}x \in \overline{\text{Orb}}([1,b]\T x_j,T)$ for some $j\neq i$. Proceeding as in the proof of Lemma \ref{lem2}, it follows that
$$I_{[\frac{1}{b},b^2]}(x,T)\subset I_{[\frac{1}{b^2},b^3]}(x_j,T),$$
hence, by \eqref{eq2thm2} again,
\begin{equation}
I_{[1,b]}(x_i,T)\subset I_{[\frac{1}{b^2},b^3]}(x_j,T) \subset I_{[\frac{1}{b^3},b^3]}(x_j,T),\label{eql}
\end{equation}
which contradicts \eqref{eqthm2}.

\medskip{}
Moreover, Lemma \ref{lem1} yields, for any non-zero polynomial $p$,
$$X=\bigcup _{i=1}^N\overline{p(T)\left(\overline{\text{Orb}}([1,b]\T x_i,T)\right)}=\bigcup _{i=1}^N\overline{\text{Orb}}([1,b]\T p(T)(x_i),T),$$
so that, by minimality of $N$, $I_{[1,b]}(p(T)(x_i),T)\neq \emptyset$ for every $i=1,\ldots,N$. Therefore
\begin{eqnarray*}
\Span(\text{Orb}([1,b]\T x_1,T))\setminus \{0\} & = & \bigcup _{p\neq 0}\orb([1,b]\T p(T)(x_1),T)\\
& \subset & \bigcup _{p\neq 0}I_{[\frac{1}{b},b^2]}(p(T)(x_1),T)\quad \text{ by (\ref{eq3})}\\
& \subset & \bigcup _{p\neq 0}I_{[1,b^3]}(\frac{p(T)(x_1)}{b},T)\\
& \subset & \bigcup _{i=1}^N I_{[\frac{1}{b^3},b^6]}(x_i,T)\quad \text{ by Lemma \ref{lem2}}\\
& \subset & \bigcup _{i=1}^NI_{[\frac{1}{b^6},b^6]}(x_i,T).\\
\end{eqnarray*}
Since $\Span(\text{Orb}([1,b]\T x_1,T))\setminus \{0\}$ is connected and since, by (\ref{eqthm2}),
$$I_{[\frac{1}{b^6},b^6]}(x_i,T)\cap I_{[\frac{1}{b^6},b^6]}(x_j,T)=\emptyset,$$
for any $i\neq j$, it follows that there exists $1\leq i\leq N$ such that
$$\Span(\text{Orb}([1,b]\T x_1,T))\setminus \{0\}\subset I_{[\frac{1}{c},c]}(x_i,T)$$
for some $c\geq b$. Finally, since $\overline{\text{Orb}}([1,b]\T x_1,T)$ has non-empty interior by minimality of $N$, we get
\begin{eqnarray*}X & = & \Span\left(\overline{\text{Orb}}([1,b]\T x_1,T)\right)\\
& \subset & \overline{\Span\left(\text{Orb}([1,b]\T x_1,T)\right)}\\
& \subset & \overline{\text{Orb}}([1/c,c]\T x_i,T).
\end{eqnarray*}

\noindent{}We conclude that the set $\text{Orb}([1/c,c]\T x_i,T)$ is dense in $X$, which contradicts the minimality of $N$.
\end{proof}

\section{Characterizations independent of the ambient space}\label{Sec1}

\subsection{Hypercyclic scalar subsets}\label{hyp-sca-sets}

In this paragraph we extend the formalism and the main result of \cite{charpentier_-supercyclicity_2016} to subsets $\Gamma$ of $\ell^2(\N)$. For reasons which will become clear later, we will distinguish the case where $\Gamma\subset \C^l$, $l$ finite, from that where $\Gamma\subset \ell^2(\N)$.

\subsubsection{Hypercyclic scalar subsets of $\C^l$}We first introduce the formalism for subsets $\Gamma$ of $\C^l$, $l$ finite. Let $l\geq 1$ and $e_1,\ldots,e_l$ be the canonical basis of $\C^l$. Let also $X$ be a separable Banach space and let $x_1,\ldots,x_l \in X$ be a linearly independent family. For $\Gamma 
\subset \C^l$, we denote by $\Gamma _{x_1,\ldots ,x_l}$ the set
$$\Gamma _{x_1,\ldots ,x_l}=\left\{\sum _{i=1}^l\gamma_ix_i;\,\sum_{i=1}^l\gamma_ie_i \in \Gamma\right\}.$$
The notion of $\Gamma$-supercyclicity introduced in \cite{charpentier_-supercyclicity_2016} naturally extends as follows.
\begin{Definition}\label{def-Gamma-hyp}Let $\Gamma \subset \C^l$, $l\geq 1$. We say that $T\in \LL(X)$ is $\Gamma$-supercyclic if $T$ is $\Gamma _{x_1,\ldots,x_l}$-hypercyclic in the sense of Definition \ref{def-C-hyp}, for some linearly independent family $x_1,\ldots,x_n$ of vectors in $X$, \emph{i.e.} if the set
$$\text{Orb}(\Gamma _{x_1,\ldots,x_l},T)$$
is dense in $X$.
\end{Definition}
When $\Gamma =\C$, we recover the notion of \emph{supercyclic operator} \cite{hilden_cyclic_1973}; more generally, when $\Gamma =\C^l$, $l\geq 1$, that of \emph{$l$-supercyclic operator} \cite{feldman_$n$-supercyclic_2002}. If $\Gamma \subset \C$ is not empty and not reduced to $\{0\}$, any hypercyclic operator is trivially $\Gamma$-supercyclic. The fact that this holds for $\Gamma \subset \C^l$, $\Gamma\setminus \{0\}\neq \emptyset$, with $l>1$ requires a brief explanation.
\begin{proposition}\label{prop-hyp-Gammahyp}Let $\Gamma \subset \C^l$, $l\geq 1$, be such that $\Gamma \setminus \{0\}$ is non-empty. Then, any hypercyclic operator is $\Gamma$-supercyclic.
\end{proposition}
\begin{proof}Let $x\in X$ be a hypercyclic vector for $T$ and let $(x_1,\ldots,x_l)$ be any linearly independent family in $X$. We pick any $z\in \Gamma_{x_1,\ldots,x_l}$ and consider an isomorphism of $X$ mapping $z$ to $x$. Then $(F(x_1),\ldots,F(x_l))$ is linearly independent and
$$x\in F\left(\Gamma_{x_1,\ldots,x_l}\right)=\Gamma_{F(x_1),\ldots,F(x_l)},$$
hence $T$ is $\Gamma_{F(x_1),\ldots,F(x_l)}$-supercyclic.
\end{proof}

\begin{remark}The previous proposition points out an important difference between the understanding of \emph{$\Gamma$-supercyclicity} and \emph{$C$-hypercyclicity} according to Definitions \ref{def-Gamma-hyp} and \ref{def-C-hyp}. Indeed, the first definition is independent of the ambient space while the second one clearly depends on the space. As a consequence, while any hypercyclic operator is $\Gamma$-supercyclic ($\Gamma\setminus \{0\}\neq \emptyset$), there obviously exist non-empty subsets $C$ of $X$ and hypercyclic operators on $X$ which are not $C$-hypercyclic. Nevertheless if $C$ is a subset of $X$ with $C\setminus \{0\}\neq \emptyset$ then doing like in the proof of Proposition \ref{prop-hyp-Gammahyp} one may remark that every hypercyclic operator is conjugate to a $C$-hypercyclic operator.
\end{remark}

We now extend the definition of \emph{hypercyclic scalar subset} introduced in \cite{charpentier_-supercyclicity_2016} to subsets of $\C^l$.

\begin{Definition}\label{def-hyp-set-C}Let $\Gamma \subset \C^l$, $l\geq 1$. $\Gamma$ is said to be a \emph{hypercyclic scalar subset} if for every separable infinite dimensional complex Banach space $X$, any $\Gamma$-supercyclic operator on $X$ is hypercyclic (or, equivalently, if for every $X$ and every linearly independent family $x_1,\ldots,x_l\in X$, every $\Gamma_{x_1,\ldots,x_l}$-hypercyclic operator $T\in \LL(X)$ is hypercyclic).
\end{Definition}

\begin{remark}In view of Proposition \ref{prop-hyp-Gammahyp}, if $\Gamma$ is a hypercyclic scalar subset, then an operator $T\in\LL(X)$ is hypercyclic if and only if it is $\Gamma$-supercyclic.
\end{remark}

At this point, we shall remark that the definition of the sets $\Gamma _{x_1,\ldots,x_l}$ apparently depends on the choice of the canonical basis (which fixes the coordinates of points in $\Gamma$), and then both definitions of $\Gamma$-supercyclicity and hypercyclic scalar subsets may depend also on that choice. Fortunately, this is not the case. Indeed, a simple algebraic computation gives the following.

\begin{proposition}\label{Prop_depend_base}
	Let $\Gamma \subset \C^l$, $l\geq 1$, $f:=(f_1,\ldots,f_l)$ be a basis of $\C^l$ and $(x_1,\ldots,x_l)$ a linearly independent family in $X$. We denote by $F$ the isomorphism from $\C^l$ to $\Span(x_1,\ldots,x_l)$ such that $F(e_i)=x_i$ and set $z_i:=F(f_i)$ for every $1\leq i\leq l$. Then
	\[
	\Gamma _{x_1,\ldots,x_l}=\Gamma^f_{z_1,\ldots,z_l},
	\]
	where $\Gamma^f_{z_1,\ldots,z_l}:=\{\sum _{i=1}^l\gamma^f_iz_i,\,\sum _{i=1}^l\gamma^f_if_i \in \Gamma\}$. In particular,
	\begin{enumerate}
		\item $T\in \LL(X)$ is $\Gamma$-supercyclic if and only if there exists a (or equivalently \emph{for any}) basis $f:=(f_1,\ldots,f_l)$ of $\C^l$, there exists a linearly independent family $(z_1,\ldots,z_l)$ in $X$ such that $T$ is $\Gamma^f_{z_1,\ldots,z_l}$-hypercyclic;
		\item $\Gamma$ is a hypercyclic scalar subset if and only if for any $X$, for any (or equiv. \emph{some}) basis $f:=(f_1,\ldots,f_l)$ of $\C^l$, and any linearly independent family $(z_1,\ldots,z_l)$ in $X$, every $\Gamma^f_{z_1,\ldots,z_l}$-hypercyclic operator $T\in \LL(X)$ is hypercyclic.
	\end{enumerate}
\end{proposition}
Let us reformulate the main result of \cite{charpentier_-supercyclicity_2016}.
\begin{theorem}[Theorem A of \cite{charpentier_-supercyclicity_2016}]A subset $\Gamma$ of $\C$ is a hypercyclic scalar subset if and only if $\Gamma\setminus\{0\}$ is non-empty and contained in a vector annulus (see Definition \ref{vector-annulus}).
\end{theorem}

Our extension to any $l$ finite reads as follows.
\begin{theorem}\label{thm-main}Let $l\geq 1$. A subset $\Gamma$ of $\C^l$ is a hypercyclic scalar subset if and only if $\Gamma\setminus\{0\}$ is non-empty and contained in a finite union of vector annuli.
\end{theorem}

Note that if $\Gamma' \subset \Gamma $ and $\Gamma$ is a hypercyclic scalar subset, then $\Gamma '$ is also a hypercyclic scalar subset, as well as $\mu \Gamma$ for any non-zero complex number $\mu$. Moreover it is straightforward to check that $\Gamma \subset \C^l$ is contained in a finite union of vector annuli if and only if so is $\Gamma_{x_1,\ldots,x_l}$ for any linearly independent family $x_1,\ldots,x_l$ of $X$. Thus the "if" part of Theorem \ref{thm-main} has already been proven, this is Theorem \ref{thm2}. The proof of the necessary part is postponed to Paragraph \ref{pf-thm-Cl-finite}.

\subsubsection{Hypercyclic scalar subsets of $\ell^2(\N)$}\label{BF-sca-sets}The formalism of the previous paragraph makes sense in any Fr\'echet space $X$ (or more generally in any topological vector space), because any finite dimensional subspace of such $X$ is isomorphic to $\C^l$ for some $l$. If one wants to extend this formalism to \emph{infinite} dimensional subsets $\Gamma$, the most natural way is to restrict ourselves to separable Hilbert spaces and to use the fact that any closed infinite dimensional subspace of a separable Hilbert space is isomorphic to $\ell^2(\N)$.

Let then $(e_n)_{n \geq 0}$ be the canonical basis of $\ell^2(\N)$. Let $H$ be a separable Hilbert space and let $(x_n)_{n\geq 0} \subset X$ be an orthonormal family. For $\Gamma \subset \ell^2(\N)$, we denote by $\Gamma _{(x_n)_n}$ the set
$$\Gamma _{(x_n)_n}=\left\{\sum _{i\geq 0}\gamma_ix_i;\,\sum_{i\geq 0}\gamma_ie_i \in \Gamma\right\}.$$
Definition \ref{def-Gamma-hyp} naturally extends in the following way.
\begin{Definition}\label{def-Gamma-hyp-l2}Let $\Gamma \subset \ell^2(\N)$. We say that $T\in \LL(H)$ is $\Gamma$-supercyclic if $T$ is $\Gamma _{(x_n)_n}$-hypercyclic in the sense of Definition \ref{def-C-hyp}, for some orthonormal family $(x_n)_n$ in $H$, \emph{i.e.} if the set
$$\text{Orb}(\Gamma _{(x_n)_n},T)$$
is dense in $X$.
\end{Definition}
\noindent{}As in the case of subsets of $\C^l$, the following easily holds true. We omit the proof.
\begin{proposition}\label{prop-hyp-Gammahyp-l2}Let $\Gamma \subset \ell^2(\N)$ be such that $\Gamma \setminus \{0\}\neq \emptyset$. Then, any hypercyclic operator is $\Gamma$-supercyclic.
\end{proposition}

\begin{remark}As previously observed, there obviously exist non-empty subsets $C$ of a Hilbert space $H$ and hypercyclic operators on $H$ which are not $C$-hypercyclic. Yet, if $C$ is a subset of $H$ with $C\setminus \{0\}\neq \emptyset$ then every hypercyclic operator on $H$ is conjugate to a $C$-hypercyclic operator.
\end{remark}

We can now extend the notion of \emph{hypercyclic scalar subset} to subsets of $\ell^2(\N)$.

\begin{Definition}\label{def-hyp-set-C-l2}Let $\Gamma\subset \ell^2(\N)$. $\Gamma$ is said to be a \emph{hypercyclic scalar subset} if for every separable infinite dimensional complex Hilbert space $H$, any $\Gamma$-supercyclic operator on $H$ is hypercyclic (or, equivalently, if for every separable infinite dimensional complex Hilbert space $H$ and every orthonormal family $(x_n)_n \subset X$, every $\Gamma_{(x_n)_n}$-hypercyclic operator $T\in \LL(H)$ is hypercyclic).
\end{Definition}

\begin{remark}In view of Proposition \ref{prop-hyp-Gammahyp-l2}, if $\Gamma$ is a hypercyclic scalar subset, then an operator $T\in\LL(H)$ is hypercyclic if and only if it is $\Gamma$-supercyclic.
\end{remark}

Here again, the definition of the sets $\Gamma _{(x_n)_n}$ depends on the choice of the Hilbert basis of $\ell^2(\N)$, and then both definitions of $\Gamma$-supercyclicity and hypercyclic scalar subsets may depend also on that choice. For the same reasons as in the finite dimensional setting, this is not the case and we have:

\begin{proposition}\label{Prop_depend_baseH}
	Let $\Gamma \subset \ell^2(\N)$, $f:=(f_n)_n$ be an orthonormal basis of $\ell^2(\N)$ and $(x_n)_n$ an orthonormal family in $H$. We denote by $F$ the isomorphism from $\ell^2(\N)$ to $\overline{\Span}((x_n)_n)$ such that $F(e_i)=x_i$ and set $z_i:=F(f_i)$ for every $i\in\N$. Then
	\[
	\Gamma _{(x_n)_n}=\Gamma^f_{(z_n)_n},
	\]
	where $\Gamma^f_{(z_n)_n}:=\{\sum _{i\geq 0}\gamma^f_i z_i,\,\sum _{i\geq 0}\gamma^f_i f_i \in \Gamma\}$. In particular,
	\begin{enumerate}
		\item $T\in \LL(H)$ is $\Gamma$-supercyclic if and only if there exists an (or equivalently \emph{for any}) orthonormal basis $f$ of $\ell^2(\N)$, there exists an orthonormal family $(z_n)_n$ in $H$ such that $T$ is $\Gamma^f_{(z_n)_n}$-hypercyclic;
		\item $\Gamma$ is a hypercyclic scalar subset if and only if for any complex separable Hilbert space $H$, for any (or equiv. \emph{some}) orthonormal basis $f$ of $\ell^2(\N)$, and any orthonormal family $(z_n)_n$ in $H$, every $\Gamma^f_{(z_n)_n}$-hypercyclic operator $T\in \LL(H)$ is hypercyclic.
	\end{enumerate}
\end{proposition}
We get a complete characterization of hypercyclic scalar subsets of $\ell^2(\N)$ as follows.
\begin{theorem}\label{thm-main-l2} A subset $\Gamma$ of $\ell^2(\N)$ is a hypercyclic scalar subset if and only if $\Gamma\setminus\{0\}$ is non-empty and contained in a finite union of vector annuli.
\end{theorem}

Here again, the sufficiency is Theorem \ref{thm2}. The proof of the necessary part is given in Paragraph \ref{pf-thm-l2}.

\subsection{Proof of the "only if" part of Theorem \ref{thm-main}}\label{pf-thm-Cl-finite}

We aim to prove the following.

\begin{proposition}\label{only-if-main}If $\Gamma \subset \C^l$ is not contained in a finite union of vector annuli, then there exists a Hilbert space $H$ and a $\Gamma$-supercyclic operator $T$ on $H$ which is not hypercyclic.
\end{proposition}

We need two lemmas. In the sequel, we will say that a sequence $(x_n)_n$ in a vector space consists of \emph{pairwise independent vectors} if $x_n$ and $x_m$ are linearly independent, whenever $n\neq m$.
\begin{lemma}\label{lemred}Let $\Gamma$ be a subset of $\C^l$ which is not contained in a finite union of one dimensional vector subspaces. Then there exist a sequence $(\lambda_k)_{k\in\N}\subset\Gamma$ of pairwise independent vectors, a basis $f=\{f_1,\ldots,f_l\}$ of $\C^l$, and $L\in\{1,\ldots,l\}$, such that if we denote by $\lambda_{k}^{(i)}$ the $i$-th coordinate of $\lambda_k$ with respect to $f$, then for every $k\in\N$:\label{Cc1}
	\begin{enumerate}
		\item for every $1\leq i\leq L$, $\lambda_{k}^{(i)}\neq 0$\label{C3};
		\item for every $L+1\leq i\leq l$, $\lambda_{k}^{(i)}=0$\label{C4};
		\item for every $2\leq i\leq l$, $\lim_{k\to\infty}\lambda_{k}^{(i)}=0$ if $\Gamma$ is bounded\label{C2};
		\item $\lim_{k\to\infty}\lambda_{k}^{(1)}=\infty$\label{C2i} if $\Gamma$ is unbounded.
	\end{enumerate}
\end{lemma}

For the next lemma, we will use the notation $\Gamma_{x_1,\ldots,x_L}^f$ introduced in Proposition \ref{Prop_depend_base}.

\begin{lemma}\label{lemce}Let $\Gamma=(\lambda _n)_n\subset \C^l$, $l\geq 1$, and let $f:=\{f_m,m\in I_1\}\cup \{f_m,m\in I_2\}$ be a basis of $\C^l$ such that $\{f_m,m\in I_1\}$ is a basis of $\Span(\lambda_n,\,n\geq 0)$. We denote by $\lambda _n^{(i)}$ the $i$-th coordinate of the $n$-th sequence $\lambda_n$, with respect to $f$. We assume that $\lambda _n^{(i)}\neq 0$ for any $n\geq 0$ and any $i\in I_1$, and that
$$\lim_{n\to\infty}|\lambda_{n}^{(i_0)}|=0\text{ or }\infty$$
for some $i_0\in I_1$. Then there exist a separable Hilbert space $H$, a non-hypercyclic operator $T$ on $H$ and a linearly independent family $(z_1,\ldots,z_l)_i$ in $H$ such that $T$ is $\Gamma^{f}_{z_1,\ldots,z_l}$-hypercyclic.
\end{lemma}

Let us briefly admit the two previous lemmas and deduce the proof of Proposition \ref{only-if-main}.

\begin{proof}[Proof of Proposition \ref{only-if-main}]Observe first that if there exists a family $\{g_1,\ldots,g_N\}\subset \C^l$ such that $\Gamma= \cup _{i=1}^N\Gamma _i g_i$ with some $\Gamma _{i_0} \setminus \{0\}$ not bounded or not bounded away from $0$, then we are back to the $1$-dimensional case treated in \cite{charpentier_-supercyclicity_2016}: we can exhibit a non-hypercyclic operator $T$ on $\ell ^2(\N)$ or $\ell^2(\Z)$ which is $\Gamma _{i_0}$-supercyclic, and a fortiori $\Gamma$-supercyclic.

Let us now assume that we are not in the previous situation and that $\Gamma$ is not included in a finite union of complex lines. By Lemma \ref{lemred}, there exist a sequence $(\lambda_n)_n\subset \Gamma$ consisting of pairwise independent vectors, a basis $f:=(f_1,\ldots,f_l)$ of $\C^l$, and $L\in \{1,\ldots,l\}$, such that items (1) and (2) hold, and (3) or (4) as well. Thus, the assumptions of Lemma \ref{lemce} are satisfied and we deduce that there exist a separable Hilbert space $H$, a non-hypercyclic $T\in \mathcal{L}(H)$ and a linearly independent family $(z_1,\ldots,z_l)$ in $H$ such that $T$ is $\Gamma^{f}_{z_1,\ldots,z_l}$-hypercyclic. By Proposition \ref{Prop_depend_base}, $(\lambda_n)_n$ is not a hypercyclic scalar subset, and $\Gamma$ is not either.
\end{proof}

\medskip{}
We now come back to the proofs of Lemma \ref{lemred} and \ref{lemce}.

\begin{proof}[Proof of Lemma \ref{lemred}]Let $\Gamma\subset \C^l$ be not contained in a finite union of one complex lines. If $\Gamma$ is bounded, then it is relatively compact and there exist a non-zero vector $f_1\in \C^l$ and a sequence $(\lambda_k)_k \subset \Gamma \setminus \C f_1$ of pairwise independent vectors, such that the distance from $\lambda_k$ to $\C f_1$ tends to $0$ as $k\rightarrow \infty$. Let us complete $f_1$ to form a basis $\{f_1,\ldots,f_l\}$ of $\C^l$ and for every $k\in\N$, let $\lambda_{k}^{(i)}$ denote the coordinates of $\lambda_k$ in this basis. Since, the distance from $\lambda_k$ to $\C f_1$ tends to $0$ as $k\rightarrow \infty$ then (\ref{C2}) is satisfied. Moreover, conditions (\ref{C3}) and (\ref{C4}) are also satisfied up to take a subsequence of $(\lambda_k)_{k\in\N}$ and to reorder $\{f_2,\ldots,f_l\}$.

The proof is similar if $\Gamma$ is unbounded. Let $(\lambda_k)_k\subset \Gamma$ be a sequence of pairwise independent vectors such that $\Vert\lambda_k\Vert\rightarrow \infty$ as $k\rightarrow \infty$. By compactness, we can assume that the sequence $(\lambda_k/\Vert\lambda_k\Vert)_k$ is convergent to some non-zero vector $f_1$ in $\C^l$, that we complete to form a basis $\{f_1,\ldots,f_l\}$ of $\C^l$. Now it is clear that the conditions (\ref{C3}), (\ref{C4}) and (\ref{C2i}) are satisfied, up to reorder $\{f_2,\ldots,f_l\}$.
\end{proof}

Let us now turn to the proof of Lemma \ref{lemce}.

\begin{proof}[Proof of Lemma \ref{lemce}]Observe first that given a linearly independent family $(z_1,\ldots,z_l)$ in a Hilbert space $H$,
$$\Gamma ^f_{z_1,\ldots,z_l}= \left\{\sum _{i=1}^l\lambda _n^{(i)}z_i,\,n\geq 0\right\}=\left\{\sum _{i\in I_1}\lambda _n^{(i)}z_i,\,n\geq 0\right\}.$$
Thus, without loss of generality, we may assume that $I_2=\emptyset$ and $I_1=\{1,\ldots,l\}$. We assume that $\lambda_{n}^{(i_0)}\rightarrow 0$, $n\rightarrow \infty$. Up to reorder $I_1$, we can also suppose that $i_0=l$.

We consider the operator $T:=B_{v}\oplus \cdots \oplus B_{v}\oplus B_{w}$ defined on $H:=\left(\ell^{2}(\Z)\right)^{l}$, where the weighted backward shifts $B_{v}$ appears $l-1$ times, and the weights $v$ and $w$ are given respectively by
	
$$v_i=\begin{cases}
	2& \text{ if }i>0\\
	\frac{1}{2}& \text{ if }i\leq 0
	\end{cases}\quad \text{and}\quad w_i=\begin{cases}
	2& \text{ if }i>0\\
	1& \text{ if }i\leq 0
	\end{cases}.$$
	
\noindent{}Let $(y_{k}^{(1)})_{k\in\N},\ldots,(y_{k}^{(l)})_{k\in\N}$ be $l$ sequences in $c_{00}(\Z)$ such that the sequence ${(y_{k}^{(1)},\ldots,y_{k}^{(l)})}_{k\in\N}$ is dense in $H=\left(\ell^{2}(\Z)\right)^l$. We also define the degree $d_k$ of an element of this sequence by $d_k:=\max(\vert j\vert:\exists s\leq k,\exists 1\leq i\leq l: y_{s}^{(i)}(j)\neq 0)$. In what follows, we denote by $F_\frac{1}{v}$ (resp. $F_\frac{1}{w}$) the weighted forward shift being the inverse of $B_v$ (resp. $B_w$). In other words, $F_{\frac{1}{v}}$ is the weighted forward shift $F_\nu$ where $\nu_i=\begin{cases}
	\frac{1}{2}\text{ if } i\geq 0\\
	2\text{ if } i<0
	\end{cases}$ and $F_{\frac{1}{w}}$ is the weighted forward shift $F_\omega$ where $\omega_i=\begin{cases}
	\frac{1}{2}\text{ if } i\geq 0\\
	1\text{ if } i<0
	\end{cases}$.
Remark that $T$ is not hypercyclic on $H$, since $B_w$ is expansive and so not hypercyclic on $\ell^{2}(\Z)$.

From now on, we construct by induction $l$ subsequences $(\lambda_{\varphi(k)}^{(i)})_{k\in\N}$ of $(\lambda_{k}^{(i)})_{k\in\N}$, $1\leq i\leq l$, and an increasing sequence $(m_k)_{k\in\N}\subset\N$ satisfying:
	\begin{enumerate}[(i)]
		\item $\left\Vert\frac{1}{\lambda_{\varphi(k)}^{(i)}} F_{\frac{1}{v}}^{m_k}y_{k}^{(i)}\right\Vert<2^{-k}$ for every $1\leq i < l$\label{nbcritIa};
		\item $\left\Vert\frac{1}{\lambda_{\varphi(k)}^{(l)}} F_{\frac{1}{w}}^{m_k}y_{k}^{(l)}\right\Vert<2^{-k}$\label{nbcritIb};
		\item $\left\Vert\frac{\lambda_{\varphi(j)}^{(i)}}{\lambda_{\varphi(k)}^{(i)}}F_{\frac{1}{v}}^{m_k-m_j}y_{k}^{(i)}\right\Vert<2^{-k}$ for every $1\leq i < l$ and every $j<k$\label{nbcritIIa};
		\item $\left\Vert\frac{\lambda_{\varphi(j)}^{(l)}}{\lambda_{\varphi(k)}^{(l)}}F_{\frac{1}{w}}^{m_k-m_j}y_{k}^{(l)}\right\Vert<2^{-k}$\label{nbcritIIb};
		\item $\left\Vert\frac{\lambda_{\varphi(k)}^{(i)}}{\lambda_{\varphi(j)}^{(i)}}B_{v}^{m_k-m_j}y_{j}^{(i)}\right\Vert<2^{-k}$ for every $1\leq i < l$ and every $j<k$\label{nbcritIIIa};
		\item $\left\Vert\frac{\lambda_{\varphi(k)}^{(l)}}{\lambda_{\varphi(j)}^{(l)}}B_{w}^{m_k-m_j}y_{j}^{(l)}\right\Vert<2^{-k}$ for every $j<k$.\label{nbcritIIIb}
	\end{enumerate}
Once this construction has been made until step $k-1$, since for every $j<k$ and every $m>0$,
\[\Vert B_{w}^{m-m_j}y_{j}^{(l)}\Vert<2^{d_k+1}\max_{s\leq k}\Vert y_{s}^{(l)}\Vert,
\]
it suffices to choose $\varphi(k)$ large enough so that $\lambda_{\varphi(k)}^{(l)}$ is a small enough element of the converging to $0$ sequence $(\lambda_{n}^{(l)})_{n\in\N}$ to ensure that (\ref{nbcritIIIb}) is satisfied. 
Moreover, one can choose $m_k$ sufficiently large for (\ref{nbcritIa}), (\ref{nbcritIb}), (\ref{nbcritIIa}), (\ref{nbcritIIb}) and (\ref{nbcritIIIa}) to hold, since for every $y\in c_{00}(\Z)$,
\[F_{\frac{1}{w}}^{m}y\underset{m\to+\infty}{\longrightarrow}0,\ F_{\frac{1}{v}}^{m}y\underset{m\to+\infty}{\longrightarrow}0\text{ and }B_{v}^{m}y\underset{m\to+\infty}{\longrightarrow}0.\]
This finishes the construction of the sequences $(\lambda_{\varphi(k)}^{(1)})_{k\in\N},\ldots, (\lambda_{\varphi(k)}^{(l)})_{k\in\N}$ and $(m_k)_{k\in\N}$. 

Now, thanks to (\ref{nbcritIa}) and (\ref{nbcritIb}) we define the orthogonal family $(\tilde{z_1},\ldots,\tilde{z_l})$ in $H$ by
\begin{equation*}\tilde{z_i}:=\begin{cases}
	\left(0,\ldots,0,\sum_{j=0}^{\infty}\frac{1}{\lambda_{\varphi(j)}^{(i)}} F_{\frac{1}{v}}^{m_j}y_{j}^{(i)},0,\ldots,0\right)\text{ if } 1\leq i <l,\\
	\left(0,\ldots,0,\sum_{j=0}^{\infty}\frac{1}{\lambda_{\varphi(j)}^{(l)}} F_{\frac{1}{w}}^{m_j}y_{j}^{(l)}\right)\text{ if } i=l,
	\end{cases}
\end{equation*}
where the only non-zero coordinate of $\tilde{z_i}$ is at position $i$. We claim that $T$ is $\Gamma_{\tilde{z_1},\ldots,\tilde{z_L}}^{f}$-hypercyclic. Indeed, for $k\in\N$, we have
	
\begin{align*}
& \left\Vert T^{m_k}(\sum_{i=1}^{l}\lambda_{\varphi(k)}^{(i)}\tilde{z_i})-(y_{k}^{(1)},\ldots,y_{k}^{(l)})\right\Vert ^2 & & \\
& =\left\Vert \bigoplus_{i=1}^{l-1}(\lambda_{\varphi(k)}^{(i)} B_{v}^{m_k}\sum_{j=0}^{\infty}\frac{1}{\lambda_{\varphi(j)}^{(i)}} F_{\frac{1}{v}}^{m_j}y_{j}^{(i)}-y_{k}^{(i)} )\bigoplus (\lambda_{\varphi(k)}^{(l)} B_{w}^{m_k}\sum_{j=0}^{\infty}\frac{1}{\lambda_{\varphi(j)}^{(l)}} F_{\frac{1}{w}}^{m_j}y_{j}^{(i)}-y_{k}^{(l)})\right\Vert ^2 & & \\
& =\sum_{i=1}^{l-1}\left\Vert \lambda_{\varphi(k)}^{(i)} B_{v}^{m_k}\sum_{j=0}^{\infty}\frac{1}{\lambda_{\varphi(j)}^{(i)}} F_{\frac{1}{v}}^{m_j}y_{j}^{(i)}-y_{k}^{(i)}\right\Vert ^2+\left\Vert \lambda_{\varphi(k)}^{(l)} B_{w}^{m_k}\sum_{j=0}^{\infty}\frac{1}{\lambda_{\varphi(j)}^{(l)}} F_{\frac{1}{w}}^{m_j}y_{j}^{(l)}-y_{k}^{(l)}\right\Vert ^2. & &\\
\end{align*}
	Now using (\ref{nbcritIIa}) and (\ref{nbcritIIIa}), we get for every $1\leq i < l$,
	
	\begin{align*}
& \left\Vert \lambda_{\varphi(k)}^{(i)} B_{v}^{m_k}\sum_{j=0}^{\infty}\frac{1}{\lambda_{\varphi(j)}^{(i)}} F_{\frac{1}{v}}^{m_j}y_{j}^{(i)} -y_{k}^{(i)}\right\Vert & &\\
& \leq\left\Vert \sum_{j<k}\frac{\lambda_{\varphi(k)}^{(i)}}{\lambda_{\varphi(j)}^{(i)}} B_{v}^{m_k}F_{\frac{1}{v}}^{m_j}y_{j}^{(i)}\right\Vert+\left\Vert\frac{\lambda_{\varphi(k)}^{(i)}}{\lambda_{\varphi(k)}^{(i)}} B_{v}^{m_k}F_{\frac{1}{v}}^{m_k}y_{k}^{(i)}-y_{k}^{(i)}\right\Vert+\left\Vert \sum_{j>k}\frac{\lambda_{\varphi(k)}^{(i)}}{\lambda_{\varphi(j)}^{(i)}} B_{v}^{m_k}F_{\frac{1}{v}}^{m_j}y_{j}^{(i)}\right\Vert & &\\
& \leq \sum_{j<k}\left\Vert\frac{\lambda_{\varphi(k)}^{(i)}}{\lambda_{\varphi(j)}^{(i)}} B_{v}^{m_k-m_j}y_{j}^{(i)}\right\Vert+ \sum_{j>k}\left\Vert\frac{\lambda_{\varphi(k)}^{(i)}}{\lambda_{\varphi(j)}^{(i)}} F_{\frac{1}{v}}^{m_j-m_k}y_{j}^{(i)}\right\Vert & &\\
& \leq \sum_{j<k}\frac{1}{2^k}+\sum_{j>k}\frac{1}{2^j}\leq \frac{k+1}{2^k}.\quad\quad\quad\quad\quad\quad\quad\quad\quad\quad\quad
\end{align*}
On the other hand, similar computations with (\ref{nbcritIIb}) and (\ref{nbcritIIIb}) yield
\[
\left\Vert \lambda_{\varphi(k)}^{(l)} B_{w}^{m_k}\sum_{j=0}^{\infty}\frac{1}{\lambda_{\varphi(j)}^{(l)}} F_{\frac{1}{w}}^{m_j}y_{j}^{(i)}-y_{k}^{(l)}\right\Vert\leq \frac{k+1}{2^k}.
\]
Altogether we obtain, for every $k\in\N$,	
\[
\left\Vert T^{m_k}(\sum_{i=1}^{l}\lambda_{\varphi(k)}^{(i)}\tilde{z_i})-(y_{k}^{(1)},\ldots,y_{k}^{(l)})\right\Vert \leq l\frac{k+1}{2^k}\underset{k\to+\infty}{\longrightarrow}0.
\]
Since the sequence ${(y_{k}^{(1)},\ldots,y_{k}^{(l)})}_{k\in\N}$ is dense in $H$, we conclude that $T$ is $\Gamma_{\tilde{z_1},\ldots,\tilde{z_l}}^{f}$-hypercyclic. Finally setting $z_i=\frac{\tilde{z_i}}{\left\Vert \tilde{z_i}\right\Vert}$ for $1\leq i \leq l$, consider an isomorphism $S$ of $H$ sending $\tilde{z_i}$ to $z_i$. Then, it is plain that the operator $S\circ T\circ S^{-1}$ is $\Gamma_{z_1,\ldots,z_l}^{f}$-hypercyclic.

The proof works likewise under the assumption $|\lambda _n^{(i_0)}|\rightarrow \infty$, up to changing the weight $w$. We refer the reader to the more technical proof of Theorem \ref{thm-main-l2} (see next paragraph), where all the cases will be treated in details.
	
\end{proof}

\begin{remark}In the previous proof, the family $(\tilde{z_1},\ldots,\tilde{z_l})$ is orthogonal, and can be chosen as an \emph{orthonormal} family, up to a finite number of dilations. Therefore $\Gamma \subset \C^l$ is a hypercyclic scalar subset in the sense of Definition \ref{def-hyp-set-C} if and only if it is a hypercyclic scalar subset \emph{with respect to every complex separable Hilbert space and any orthonormal family}; more precisely if and only if for every \emph{complex separable Hilbert space} $H$ and every \emph{orthonormal family} $(z_1,\ldots,z_l)$, every $\Gamma_{z_1,\ldots,z_l}$-hypercyclic operator on $H$ is hypercyclic.
\end{remark}

\subsection{Proof of the "only if" part of Theorem \ref{thm-main-l2}}\label{pf-thm-l2}

We intend to prove the following analogue to Proposition \ref{only-if-main}.

\begin{proposition}\label{only-if-main-l2}If $\Gamma \subset \ell^2(\N)$ is not contained in a finite union of vector annuli, then there exists a Hilbert space $H$ and a $\Gamma$-supercyclic operator $T$ on $H$ which is not hypercyclic.
\end{proposition}

According to Proposition \ref{only-if-main}, we need only to prove the previous proposition for subsets $\Gamma$ of $\ell^2(\N)$ which are not contained in any finite dimensional subspace of $\ell^2(\N)$. Thus we are reduced to prove the following.

\begin{proposition}\label{only-if-main-l2-reduced}Let $\Gamma=(\lambda_n)_{n\geq 0}\subset \ell^2(\N)$ be a linearly independent family. There exist a separable infinite dimensional complex Hilbert space $H$ and a $\Gamma$-supercyclic operator $T$ on $H$ which is not hypercyclic.
\end{proposition}

The proof will follow the same scheme as that of Proposition \ref{only-if-main} in the finite dimensional setting. For this, we need two lemmas. The first one generalizes Lemma \ref{lemred} to any subset of $\ell^2(\N)$.

\begin{lemma}\label{lemred-inf-case}Let $\Gamma$ be a subset of $\ell^2(\N)$ which is not contained in a finite union of vector annuli. Then, from any linearly independent family in $\Gamma$, we can extract a linearly independent sequence $(\lambda_k)_k$ such that there exist an orthonormal basis $(f_m)_m$ of $\ell^2(\N)$, a partition $\N=I_1\cup I_2$ and some $i_0\in I_1$, such that if we write $\lambda_k=\sum_{m\geq 0}\lambda_{k}^{(m)}f_m$, then:
	\begin{enumerate}
		\item for every $k\in\N$ and every $m\in I_1$, $\lambda_{k}^{(m)}\neq 0$;
		\item for every $k\in\N$ and every $m\in I_2$, $\lambda_{k}^{(m)}=0$;
		\item $|\lambda_{k}^{(i_0)}|\rightarrow 0$ or $\infty$ as $k\rightarrow \infty$.
	\end{enumerate}
\end{lemma}

The second lemma extends Lemma \ref{lemce} to subsets of $\ell ^2(\N)$. We will use the notation $\Gamma_{(z_i)_i}^f$ introduced in Proposition \ref{Prop_depend_baseH}.

\begin{lemma}\label{lemce-inf-case}Let $\Gamma=(\lambda _n)_n\subset \ell ^2(\N)$ and let $f:=\{f_m,m\in I_1\}\cup \{f_m,m\in I_2\}$ be an orthonormal basis of $\ell ^2(\N)$ such that $\{f_m,m\in I_1\}$ is an orthonormal basis of $\overline{\Span(\lambda_n,\,n\geq 0)}$. We denote by $\lambda _n^{(i)}$ the $i$-th coordinate of the $n$-th sequence $\lambda_n$, with respect to $f$. We assume that $\lambda _n^{(i)}\neq 0$ for any $n\geq 0$ and any $i\in I_1$, and that
$$\lim_{n\to\infty}|\lambda_{n}^{(i_0)}|=0\text{ or }\infty$$
for some $i_0\in I_1$. Then there exist a separable Hilbert space $H$, a non-hypercyclic operator $T$ on $H$ and an orthonormal family $(z_i)_i$ in $H$ such that $T$ is $\Gamma^{f}_{(z_i)_i}$-hypercyclic, where
$$\Gamma^f_{(z_i)_i}:=\left\{\sum_{i\geq 0}\lambda _n^{(i)} z_i,\,n\geq 0\right\}.$$
\end{lemma}

The proofs of those lemmas are postponed in Paragraphs \ref{subsubproof1} and \ref{subsubproof2} respectively. Let us first see how we use them in order to complete the proof of Proposition \ref{only-if-main-l2-reduced}.

\begin{proof}[Proof of Proposition \ref{only-if-main-l2-reduced}]Let $\Gamma=(\lambda_n)_n$ be as in Proposition \ref{only-if-main-l2-reduced}. By Lemma \ref{lemred-inf-case}, there exist an orthonormal basis $f:=(f_m)_m$ of $\ell^2(\N)$, a partition $\N=I_1\cup I_2$ and some $i_0\in I_1$, such that items (1) and (2) hold, and (3) or (4) as well. In particular, the assumptions of Lemma \ref{lemce-inf-case} are satisfied. It follows that there exist a separable Hilbert space $H$, a non-hypercyclic $T\in \mathcal{L}(H)$ and an orthonormal family $(z_i)_i$ in $H$ such that $T$ is $\Gamma^{f}_{(z_i)_i}$-hypercyclic. By Proposition \ref{Prop_depend_baseH} we conclude that $\Gamma$ is not a hypercyclic scalar subset.
\end{proof}

Now, we have to prove Lemmas \ref{lemred-inf-case} and \ref{lemce-inf-case}.

\subsubsection{Proof of Lemma \ref{lemred-inf-case}}\label{subsubproof1}It is based on two geometric sublemmas.

\begin{lemma}\label{lem1-inf-case}Let $(x_n)_{n\geq 0}\subset \ell^2(\N)$ be a linearly independent family. Given any norm $1$ vector $f_0\in \ell^2(\N)$ with $x_n\notin \C f_0$ and $\left<x_n,f_0\right>\neq 0$ for any $n\geq 0$, there exists $f_m$, $m\geq 1$, in $F:=\overline{\Span(f_0;x_n,\,n\geq 0)}$  such that $(f_m)_{m\geq 0}$ is an orthonormal basis of $F$ and
$$\left<x_n,f_m\right>\neq 0$$
for any $n,m\geq 0$.
\end{lemma}

\begin{proof}[Proof of Lemma \ref{lem1-inf-case}]We consider an enumeration $\phi:\N\rightarrow \N$ such that for any integer $k\geq 0$ there exist infinitely many $m\geq 0$ such that $\phi(m)=k$. Let $(\varepsilon_m)_m$ be a sequence of real numbers tending to $0$ with $\varepsilon_1=\text{dist}(x_{\phi(1)},\C f_0)$. $f_0$ being given, we proceed by induction and assume that $f_1,\ldots,f_{m-1}$ have been built in such a way that, if we let $G_{i}:=\Span(f_j,\,0\leq j \leq i)$ for any $0\leq i \leq m-1$, then
\begin{enumerate}[(i)]
	\item $\Vert f_i\Vert=1$ and $\left<f_i,f_j\right>= 0$, $1\leq i\neq j\leq m-1$;\label{condi}
	\item For any $n\geq 0$, $x_n\notin G_{m-1}$;\label{condii}
	\item For any $1\leq i\leq m-1$, $\text{dist}(x_{\phi(i)},G_i) \leq \varepsilon _i$ (where $\text{dist}(x,G)$ stands for the distance from $x$ to $G$).\label{condiii}
\end{enumerate}
Let us now write $F=G_{m-1}\oplus G_{m-1}^{\perp}$ and build $f_m$ in the Hilbert space $G_{m-1}^{\perp}$. We denote by $P_{m-1}$ the orthogonal projection on $G_{m-1}^{\perp}$ and by $V_n^{m-1}:=\C P_{m-1}(x_n)$, $n\geq 0$. Note that $G_{m-1}^{\perp}$ is an infinite dimensional subspace of $F$, since $F$ is infinite dimensional while $G_{m-1}$ is finite dimensional. For any $k\geq 0$, we define
$$A_{m-1}^k:=\left\{y\in G_{m-1}^{\perp}\setminus \cup_{n\geq 0}V_n^{m-1},\,\left<y,x_k\right>\neq 0\right\}.$$
We claim that each set $A_{m-1}^k$ is a $G_{\delta}$-dense subset of the Hilbert space $G_{m-1}^{\perp}$. Indeed observe that
$$A_{m-1}^k=\bigcap_{n\geq 0}(V_n^{m-1})^c\cap \left\{y\in G_{m-1}^{\perp},\,\left<y,x_k\right>\neq 0\right\}.$$
Now, because of Property (\ref{condii}) above, any $V_n^{m-1}$, $n\geq 0$, is a proper and closed subspace of $G_{m-1}^{\perp}$. So its complement $(V_n^{m-1})^c$ is open and dense in $G_{m-1}^{\perp}$, hence by Baire Category Theorem $\cap_{n\geq 0}(V_n^{m-1})^c$ is residual in $G_{m-1}^{\perp}$. Also the set $\left\{y\in G_{m-1}^{\perp},\,\left<y,x_k\right>\neq 0\right\}$, which is the complement of the subspace orthogonal to $\C x_k$ in $G_{m-1}^{\perp}$, is open and dense in $G_{m-1}^{\perp}$. Thus, again by Baire Category Theorem, $A_{m-1}^k$ is residual in $G_{m-1}^{\perp}$. Applying another (and last) time Baire Category Theorem we get that the set
$$\bigcap _{k\geq 0}A_{m-1}^k$$
is residual in $G_{m-1}^{\perp}$, so that we can pick $\tilde{f_m}$ in $\cap _{k\geq 0}A_{m-1}^k$ such that $\Vert \tilde{f_m}-P_{m-1}(x_{\phi(m)})\Vert \leq \varepsilon _m$ and set $f_m:=\tilde{f_m}/\Vert\tilde{f_m}\Vert$.

By construction, we immediately get that $f_m$ satisfies (\ref{condi}). For (\ref{condii}) (which is the needed assumption to do the induction), let assume by contradiction that $x_n\in G_m:=\Span(f_j,\,0\leq j \leq m)$ and decompose $x_n$ in the unique way $x_n=x_n^1+P_{m-1}(x_n)$ where $x_n^1\in G_{m-1}$. Then by uniqueness of this decomposition $f_m$ must belong to $\C P_{m-1}(x_n)=V_n^{m-1}$, but this contradicts the fact that $f_m$ is in $A_{m-1}^k$. Finally, (\ref{condiii}) comes from
\begin{align}\label{eq-lem1-infinite}
\text{dist}(x_{\phi(m)},G_m) & \leq \text{dist}(P_{m-1}(x_{\phi(m)}),G_m)\nonumber\\
&\leq \text{dist}(P_{m-1}(x_{\phi(m)}),\C f_m)\\
& \leq \left\Vert \tilde{f_m} - P_{m-1}(x_{\phi(m)})\right\Vert\leq \varepsilon _m.\nonumber
\end{align}

To conclude that the family $\{f_m,\,m\geq 0\}$ satisfies the conclusion of the lemma, we only need to check that it is total in $F$, or that
$$x_n\in \overline{\Span(f_m,\,m\geq 0)}.$$
But this is straightforward from Property (\ref{condiii}) (which is now satisfied by $G_m$ for any $m\geq 1$). Indeed, let $(m_k)_k$ be an increasing sequence of integers such that $\phi(m_k)=n$ for every $k\geq 0$. As in \eqref{eq-lem1-infinite} we have
\[\text{dist}(x_n,\Span(f_m,\,m\geq 0)) \leq \text{dist}(x_{\phi(m_k)},G_{m_k})\leq \varepsilon _{m_k} \underset{k\to\infty}{\longrightarrow} 0.\]
The proof of the lemma is complete.
\end{proof}

The second lemma is as follows.

\begin{lemma}\label{lem2-inf-case}Let $(x_n)_n$ be a sequence of independent vectors in a separable Hilbert space $H$. There exist a subsequence $(x_{n_k})_k$ of $(x_n)_n$ and a vector $a\in H$ such that for any $k\geq 0$, $x_{n_k} \notin \C a$ and $\left<x_{n_k},a\right>\neq 0$, and  such that
\[
\left<x_{n_k},a\right> \rightarrow 0\text{ or }\infty,\quad \text{as }k\rightarrow \infty.
\]
\end{lemma}

\begin{proof}It is based on the following claim.
\begin{claim}Let $(x_n)_n$ be a sequence of pairwise independent vectors in $H$. There exist two orthogonal normed $1$ vectors $f_0$ and $f_1$ in $H$ such that $x_n\notin \C f_i$ and $\left<x_n,f_i\right>\neq 0$, $i=0,1$, and such that, if we denote by $P$ the orthogonal projection onto $\Span(f_0,f_1)$, the vectors $P(x_n)$ are still pairwise independent.
\end{claim}

\begin{proof}[Proof of the claim]By the Baire Category Theorem, we can choose $f_0$ in $H$, with norm $1$, such that $\left<x_n,f_0\right>\neq 0$ and $x_n\notin \C f_0$, $n\geq 0$. We have to show the existence of $f_1$ such as desired. Let us fix $n\neq m$. By assumption $x_n$ and $x_m$ are linearly independent. For $b\in f_0^{\perp}$ nonzero and $P_b$ the orthogonal projection onto $\Span(f_0,b)$, we have
\begin{equation*}P_b(x_n) = \left<x_n,f_0\right>f_0+\left<x_n,b\right>b\quad \text{and}\quad P_b(x_m) = \left<x_m,f_0\right>f_0+\left<x_m,b\right>b.
\end{equation*}
Let us denote by $A_{n,m}$ the subset of those $b$ in $f_0^{\perp}$ such that $P_b(x_n)$ and $P_b(x_m)$ are linearly independent. It corresponds to the set of $b\in f_0^{\perp}$ such that
\begin{equation*}\left<x_n,f_0\right>\left<x_m,b\right>-\left<x_m,f_0\right>\left<x_n,b\right> = \left<\left<x_n,f_0\right>x_m-\left<x_m,f_0\right>x_n,b \right>
\neq 0.
\end{equation*}
Now $A_{n,m}$ is the complement of the orthogonal to $\left<x_n,f_0\right>x_m-\left<x_m,f_0\right>x_n$ in $f_0^{\perp}$, and therefore is an open dense set of $f_0^{\perp}$ whenever $\left<x_n,f_0\right>x_m-\left<x_m,f_0\right>x_n\notin \C f_0$. Since $x_n$ and $x_m$ are linearly independent, $\left<x_n,f_0\right>x_m-\left<x_m,f_0\right>x_n$ is equal to $0$ only if $\left<x_n,f_0\right>=\left<x_m,f_0\right>=0$ which is not the case, by the choice of $f_0$. Moreover,
\[
\left<\left<x_n,f_0\right>x_m-\left<x_m,f_0\right>x_n,f_0\right>=0,
\]
hence $\left<x_n,f_0\right>x_m-\left<x_m,f_0\right>x_n\notin \C f_0$. Then, by the Baire Category Theorem, the set
\[
A:=\bigcap_{n\neq m}A_{n,m}
\]
is a dense $G_{\delta}$-subset of $f_0^{\perp}$.

Now, as we saw in the proof of Lemma \ref{lem1-inf-case}, by the choice of $f_0$ and the Baire Category Theorem, the set $B$ of those vectors $b$ in $H$ such that $b$ is orthogonal to $f_0$ and satisfies $\left<x_n,b\right>\neq 0$ and $x_n\notin \C b$ for any $n\geq 0$ is a dense $G_{\delta}$-subset of $f_0^{\perp}$. It still follows from the Baire Category Theorem that $A\cap B$ is not empty. Finally one can check that any $f_1$ in $A\cap B$ satisfies the desired property, which concludes the proof of the claim.
\end{proof}
Let us turn to the proof of the lemma. Since $(x_n)_n$ is linearly independent, we can apply the claim and pick $f_0$ and $f_1$ such that $x_n\notin \C f_i$ and $\left<x_n,f_i\right>\neq 0$, $i=1,2$, and such that $(P(x_n))_n$ is pairwise independent, where $P$ is the orthogonal projection onto $\Span(f_0,f_1)$. For $n\geq 0$, let us write
\[
x_n=P(x_n)+P_{\perp}(x_n),
\]
where $P_{\perp}$ is the orthogonal projection onto $\Span(f_0,f_1)^{\perp}$. If $\Vert P(x_n) \Vert$ is unbounded, then one may extract a subsequence $(x_{n_k})_k$ of $(x_n)_n$ such that $\left<x_n,f_0\right> \rightarrow \infty$ or $\left<x_n,f_1\right> \rightarrow \infty$, as $k\rightarrow \infty$, and then choose $a=f_0$ or $a=f_1$. If not, by compactness of bounded sets in the $2$-dimensional subspace $\Span(f_0,f_1)$, there exists a subsequence $(n_k)_k$ and $w\in \Span(f_0,f_1)$ such that $P(x_{n_k})\rightarrow w$ as $k\rightarrow \infty$. Let us pick $a$ nonzero in $\Span(f_0,f_1)$, orthogonal to $w$. It follows
\[
\left<x_{n_k},a\right>=\left<P(x_{n_k}),a\right>+\left<P_{\perp}(x_{n_k}),a\right>=\left<P(x_{n_k}),a\right>\rightarrow 0,\quad k\rightarrow \infty.
\]
Now, since $\Span(f_0,f_1)$ is $2$-dimensional, the orthogonal to $a$ in $\Span(f_0,f_1)$ is $1$-dimensional. Thus, because $(P(x_n))_n$ is pairwise linearly independent, we may assume, up to take a subsequence, that for any $k\geq 0$, $x_{n_k}\notin \C a$ and $\left<x_{n_k},a\right>=\left<P(x_{n_k}),a\right>\neq 0$. This finishes the proof of the lemma. 
\end{proof}

We can now turn to the proof of Lemma \ref{lemred-inf-case}.

\begin{proof}[Proof of Lemma \ref{lemred-inf-case}]Let $(\lambda_k)_k$ be a linearly independent sequence in $\Gamma$. By Lemma \ref{lem2-inf-case}, one gets a subsequence that we still denote $(\lambda_k)_k$ and one can pick $f_0:=a$ such that, $\left<\lambda_k,f_0\right>\neq 0$ and $\lambda_k\notin \C f_0$ for any $k\geq 0$, and $\left<\lambda_k,f_0\right>\rightarrow 0$ or $\infty$ as $k\rightarrow \infty$. By Lemma \ref{lem1-inf-case}, we can complete $f_0$ into an orthonormal basis $(f_{2m})_m$ of $\tilde{H}:=\overline{\Span(f_0,\lambda_k,\,k\geq 0)}$ , such that $\left<\lambda_k,f_{2m}\right>\neq 0$ for any $k,m\geq 0$. To finish, we complete $(f_{2m})_m$ into an orthonormal basis $(f_m)_m$ of $\ell^2(\N)$ by choosing an orthonormal basis of $(\tilde{H})^{\perp}$. It then suffices to set $I_1=\{2m,m\in \N\}$, $I_2=\{2m+1,m\in \N\}$ and $i_0=0$. 
\end{proof}

\begin{remark}The proof of Lemma \ref{lemred-inf-case}, based on that of Lemmas \ref{lem1-inf-case} and \ref{lem2-inf-case}, is actually a refinement of that of Lemma \ref{lemred}. The compactness of the finite dimensional unit sphere used in the proof of Lemma \ref{lemred} remains the very key-point in the infinite dimensional setting of Lemma \ref{lem2-inf-case}.
\end{remark}

\subsubsection{Proof of Lemma \ref{lemce-inf-case}}\label{subsubproof2}It consists in a technically involved adaptation of Lemma \ref{lemce}.

\begin{proof}[Proof of Lemma \ref{lemce-inf-case}]Let $(\lambda_{n})_{n\in\N}\in\ell^{2}(\N)$ be as in the statement of the lemma. In order to simplify the notations, we may and shall assume that $0\in I_1$ and that $i_0=0$.
Now, observe that given an orthonormal family $(z_i)_i$ in a Hilbert space $H$,
$$\Gamma ^f_{(z_i)_i}= \left\{\sum _{i\geq 0}\lambda _n^{(i)}z_i,\,n\geq 0\right\}=\left\{\sum _{i\in I_1}\lambda _n^{(i)}z_i,\,n\geq 0\right\}.$$
In other words, in the orthonormal family $(z_i)_{i\in\N}$, the elements $z_i$ with $i\in I_2$ play no role and can be chosen arbitrarily. Then, since the case $I_1$ finite has been treated in Lemma \ref{lemce}, we will also assume that $I_1=\N$.

We assume first that $\lambda_{n}^{(0)}\rightarrow 0$, $n\rightarrow \infty$. 
We consider the operator $T:=B_{w} \bigoplus_{i\geq 1} B_{v}$ defined on the $\ell^2$ direct sum of $\ell^2(\Z)$ spaces $H:=\left(\oplus_{i=0}^{\infty}\ell^{2}(\Z)\right)_{\ell^2}$, where the weighted backward shifts $B_{v}$ and $B_{w}$ are defined by weights $v$ and $w$ given respectively by

$$v_i=\begin{cases}
2& \text{ if }i>0\\
\frac{1}{2}& \text{ if }i\leq 0
\end{cases}\quad \text{and}\quad w_i=\begin{cases}
2& \text{ if }i>0\\
1& \text{ if }i\leq 0
\end{cases}.$$
Let $(y_{n})_{n\in\N}$ be a dense sequence in $H$ satisfying the following notations and properties:
\begin{enumerate}
	\item For every $n\in\N$, $y_n=(y_{n,k})_{k\geq 0}$;
	\item For every $n\in\N$, every $k\geq n$, $y_{n,k}=0$;\label{c00y}
	\item For every $n,k\in\N$ and every $\vert j\vert \geq n$, $y_{n,k}(j)=0$;
	\item For every $n,k,j\in\N$, $\vert y_{n,k}(j)\vert \leq n$.\label{Coefy}
\end{enumerate}
It should be clear that such a sequence exists in $H$. We also define the degree $d_n$ of an element of this sequence by 
\[d_n:=\max\left(\max(\vert j\vert:\exists l\leq n, \exists k\in\N: y_{l,k}(j)\neq 0),\max(k:\exists l\leq n: y_{l,k}\neq 0)\right).\]
Then, with our assumptions on the sequence $(y_n)_{n\in\N}$ we deduce that $d_n\leq n$ for every $n\in\N$. 
In what follows, we denote by $F_\frac{1}{v}$ (resp. $F_\frac{1}{w}$) the weighted forward shift being the inverse of $B_v$ (resp. $B_w$). In other words, $F_{\frac{1}{v}}$ is the weighted forward shift $F_\nu$ where $\nu_i=\begin{cases}
\frac{1}{2}\text{ if } i\geq 0\\
2\text{ if } i<0
\end{cases}$ and $F_{\frac{1}{w}}$ is the weighted forward shift $F_\omega$ where $\omega_i=\begin{cases}
\frac{1}{2}\text{ if } i\geq 0\\
1\text{ if } i<0
\end{cases}$.
Remark that $T$ cannot be hypercyclic since $B_w$ is expansive and then not hypercyclic on $\ell^{2}(\Z)$.

From now on, we construct by induction an increasing function $\varphi:\N\to\N$ and an increasing sequence $(m_k)_{k\in\N}\subset\N$ satisfying:
\begin{enumerate}[(i)]
	\item $\Vert\frac{1}{\lambda_{\varphi(k)}^{(i)}} F_{\frac{1}{v}}^{m_k}y_{k,i}\Vert<2^{-k}$ for every $i\geq 1$\label{nbcritIa};
	\item $\Vert\frac{1}{\lambda_{\varphi(k)}^{(0)}} F_{\frac{1}{w}}^{m_k}y_{k,0}\Vert<2^{-k}$\label{nbcritIb};
	\item $\Vert\frac{\lambda_{\varphi(j)}^{(i)}}{\lambda_{\varphi(k)}^{(i)}}F_{\frac{1}{v}}^{m_k-m_j}y_{k,i}\Vert<2^{-(k+i)}$ for every $i\geq 1$ and every $j<k$\label{nbcritIIa};
	\item $\Vert\frac{\lambda_{\varphi(j)}^{(0)}}{\lambda_{\varphi(k)}^{(0)}}F_{\frac{1}{w}}^{m_k-m_j}y_{k,0}\Vert<2^{-k}$ for every $j<k$\label{nbcritIIb};
	\item $\Vert\frac{\lambda_{\varphi(k)}^{(i)}}{\lambda_{\varphi(j)}^{(i)}}B_{v}^{m_k-m_j}y_{j,i}\Vert<2^{-(k+i)}$ for every $i\geq 1$ and every $j<k$\label{nbcritIIIa};
	\item $\Vert\frac{\lambda_{\varphi(k)}^{(0)}}{\lambda_{\varphi(j)}^{(0)}}B_{w}^{m_k-m_j}y_{j,0}\Vert<2^{-k}$ for every $j<k$\label{nbcritIIIb};
	\item $\left(F_{\frac{1}{v}}^{m_k}y_{k,i}\right)(0)=0$ for every $i\geq 1$\label{CritsupIIa};
	\item $\left(F_{\frac{1}{w}}^{m_k}y_{k,0}\right)(0)=0$.\label{CritsupIIb}
	\item $\left\Vert\left(\lambda_{\varphi(k)}^{(i)} \right)_i\right\Vert_2\leq \frac{2^{m_k-1}}{k}$.\label{Derargb}
\end{enumerate}
Once this construction has been made until step $k-1$, since for every $j<k$ and every $m>0$
\[\Vert B_{w}^{m-m_j}y_{j,0}\Vert<2^{d_k+1}\max_{l\leq k}\Vert y_{l,0}\Vert
,\]
then it suffices to choose $\varphi(k)$ so that $\lambda_{\varphi(k)}^{(0)}$ is a small enough element of the converging to $0$ sequence $(\lambda_{n}^{(0)})_{n\in\N}$ to ensure that (\ref{nbcritIIIb}) is satisfied. 
Moreover, one can choose $m_k$ sufficiently large for (\ref{nbcritIb}) and (\ref{nbcritIIb}) and (\ref{Derargb}) to hold, since for every $y\in c_{00}(\Z)$,
\[F_{\frac{1}{w}}^{m}y\underset{m\to+\infty}{\longrightarrow}0.\]
Furthermore, thanks to property (\ref{c00y}) of the sequence $(y_n)_{n\in\N}$ we can take $m_k$ even larger in order to satisfy (\ref{nbcritIa}), (\ref{nbcritIIa}) and (\ref{nbcritIIIa}) since for every $y\in c_{00}(\Z)$,
\[\ F_{\frac{1}{v}}^{m}y\underset{m\to+\infty}{\longrightarrow}0\text{ and }B_{v}^{m}y\underset{m\to+\infty}{\longrightarrow}0.\]
Finally, up to take $m_k$ bigger again, (\ref{CritsupIIa}) and  (\ref{CritsupIIb}) can be satisfied because $d_k\leq k$.  
This finishes the construction of the sequences $(\lambda_{\varphi(k)}^{(i)})_{k\in\N}$ for $i\geq 0$  and $(m_k)_{k\in\N}$. 

Now, thanks to (\ref{nbcritIa}) and (\ref{nbcritIb}) we define the orthogonal family $\{\tilde{z_i}\}_{i\in\N}$ in $H$ by
\[\tilde{z_0}:=\left(e_0+\sum_{j=0}^{\infty}\frac{1}{\lambda_{\varphi(j)}^{(0)}} F_{\frac{1}{w}}^{m_j}y_{j,0},0,\ldots\right)
\]
and for every $i\geq 1$
\[\tilde{z_i}:=\left(\underbrace{0,\ldots,0}_{i\text{ times}},e_0+\sum_{j=0}^{\infty}\frac{1}{\lambda_{\varphi(j)}^{(i)}} F_{\frac{1}{v}}^{m_j}y_{j,i},0,\ldots\right).
\]
We deduce from (\ref{nbcritIa}), (\ref{nbcritIb}), (\ref{CritsupIIa}) and (\ref{CritsupIIb}) that for every $i\in _N$, $1\leq\Vert \tilde{z_i}\Vert\leq 1+ \sum_{j=0}^{\infty}2^{-j}= 3$.  
Now, we normalize this family $z_i:=\frac{\tilde{z_i}}{\Vert\tilde{z_i}\Vert}$ to obtain an orthonormal family $\{z_i\}_{i\in\N}$ in $H$. 
Observe now that the diagonal operator $D:H\to H$ defined by $D\left((x_i)_{i\in\N}\right)=\left(\Vert \tilde{z_i}\Vert^{-1} x_i\right)$ is well-defined since for every $i\in\N$, $\Vert \tilde{z_i}\Vert \geq 1$, has dense range and commutes with $T$.

We claim that $T$ is $\Gamma_{(z_i)_i}^f$-hypercyclic. Thus, it suffices to prove that the orbit of $\left\{\sum_{i=0}^{\infty}\lambda_{\varphi(k)}^{(i)}z_i\right\}_{k\in\N}$ under $T$ is dense in $H$. Furthermore, from the properties of $D$, we remark that for every $m\in\N$,
\[T^m\left(\sum_{i=0}^{\infty}\lambda_{\varphi(k)}^{(i)}z_i\right)=D\circ T^m\left(\sum_{i=0}^{\infty}\lambda_{\varphi(k)}^{(i)}\tilde{z_i}\right).
\]
Then, by density of the range of $D$, we deduce that the orbit of $\left\{\sum_{i=0}^{\infty}\lambda_{\varphi(k)}^{(i)}z_i\right\}_{k\in\N}$ under $T$ is dense in $H$ whenever the orbit of $\left\{\sum_{i=0}^{\infty}\lambda_{\varphi(k)}^{(i)}\tilde{z_i}\right\}_{k\in\N}$ under $T$ is dense in $H$.
Thus, it suffices to prove that $T$ is $\Gamma_{(\tilde{z_i})_i}^f$-hypercyclic.
Let $k\in\N$,

\begin{align*}
&\left\Vert T^{m_k}\left(\sum_{i=0}^{\infty}\lambda_{\varphi(k)}^{(i)}\tilde{z_i}\right)-y_{k}\right\Vert^2\\
=&\left\Vert \left(\lambda_{\varphi(k)}^{(0)} B_{w}^{m_k}\left(e_0+\sum_{j=0}^{\infty}\frac{1}{\lambda_{\varphi(j)}^{(0)}} F_{\frac{1}{w}}^{m_j}y_{j,0}\right)-y_{k,0} \right)\bigoplus_{i=1}^{\infty} \left(\lambda_{\varphi(k)}^{(i)} B_{v}^{m_k}\left(e_0+\sum_{j=0}^{\infty}\frac{1}{\lambda_{\varphi(j)}^{(i)}} F_{\frac{1}{v}}^{m_j}y_{j,i}\right)-y_{k,i}\right)\right\Vert ^2 \\
=&\left\Vert \lambda_{\varphi(k)}^{(0)} B_{w}^{m_k}\left(e_0+\sum_{j=0}^{\infty}\frac{1}{\lambda_{\varphi(j)}^{(0)}} F_{\frac{1}{w}}^{m_j}y_{j,0}\right)-y_{k,0}\right\Vert ^2+\sum_{i=1}^{\infty}\left\Vert \lambda_{\varphi(k)}^{(i)} B_{v}^{m_k}\left(e_0+\sum_{j=0}^{\infty}\frac{1}{\lambda_{\varphi(j)}^{(i)}} F_{\frac{1}{v}}^{m_j}y_{j,i}\right)-y_{k,i}\right\Vert ^2.
\end{align*}

Now using (\ref{nbcritIIa}) and (\ref{nbcritIIIa}), we get for every $i\geq 1$,

\begin{align*}
&\left\Vert \lambda_{\varphi(k)}^{(i)} B_{v}^{m_k}\left(e_0+\sum_{j=0}^{\infty}\frac{1}{\lambda_{\varphi(j)}^{(i)}} F_{\frac{1}{v}}^{m_j}y_{j,i}\right)-y_{k,i}\right\Vert \\
\leq&\left\Vert\lambda_{\varphi(k)}^{(i)} B_{v}^{m_k}(e_0)\right\Vert+
\left\Vert \sum_{j<k}\frac{\lambda_{\varphi(k)}^{(i)}}{\lambda_{\varphi(j)}^{(i)}} B_{v}^{m_k}F_{\frac{1}{v}}^{m_j}y_{j,i}\right\Vert+\left\Vert\frac{\lambda_{\varphi(k)}^{(i)}}{\lambda_{\varphi(k)}^{(i)}} B_{v}^{m_k}F_{\frac{1}{v}}^{m_k}y_{k,i}-y_{k,i}\right\Vert+\left\Vert \sum_{j>k}\frac{\lambda_{\varphi(k)}^{(i)}}{\lambda_{\varphi(j)}^{(i)}} B_{v}^{m_k}F_{\frac{1}{v}}^{m_j}y_{j,i}\right\Vert\\
\leq& \left\Vert\lambda_{\varphi(k)}^{(i)} B_{v}^{m_k}(e_0)\right\Vert+\sum_{j<k}\left\Vert\frac{\lambda_{\varphi(k)}^{(i)}}{\lambda_{\varphi(j)}^{(i)}} B_{v}^{m_k-m_j}y_{j,i}\right\Vert+ \sum_{j>k}\left\Vert\frac{\lambda_{\varphi(k)}^{(i)}}{\lambda_{\varphi(j)}^{(i)}} F_{\frac{1}{v}}^{m_j-m_k}y_{j,i}\right\Vert\\
\leq& \left\vert\lambda_{\varphi(k)}^{(i)} \frac{1}{2^{m_k-1}}\right\vert+\sum_{j<k}\frac{1}{2^{k+i}}+\sum_{j>k}\frac{1}{2^{j+i}}\\
\leq& \left\vert\lambda_{\varphi(k)}^{(i)} \frac{1}{2^{m_k-1}}\right\vert+\frac{k+1}{2^k}\times\frac{1}{2^i}.
\end{align*}
On the other hand, similar computations with (\ref{nbcritIIb}) and (\ref{nbcritIIIb}) yield,

\[\left\Vert \lambda_{\varphi(k)}^{(0)} B_{w}^{m_k}\left(e_0+\sum_{j=0}^{\infty}\frac{1}{\lambda_{\varphi(j)}^{(0)}} F_{\frac{1}{w}}^{m_j}y_{j,0}\right)-y_{k,0}\right\Vert\leq 2\vert \lambda_{\varphi(k)}^{(0)}\vert+\frac{k+1}{2^k}.\]
Altogether we obtain, for every $k\in\N$,

\begin{align*}
\left\Vert T^{m_k}\left(\sum_{i=0}^{\infty}\lambda_{\varphi(k)}^{(i)}\tilde{z_i}\right)-y_{k}\right\Vert &\leq \sqrt{\left(2\vert \lambda_{\varphi(k)}^{(0)}\vert+\frac{k+1}{2^k}\right)^2+\left\Vert\left( \frac{k+1}{2^k}\times\frac{1}{2^i}+\vert\lambda_{\varphi(k)}^{(i)}\vert \frac{1}{2^{m_k-1}}\right)_{i\geq 1}\right\Vert_2^2}\\
&\leq 2\vert \lambda_{\varphi(k)}^{(1)}\vert+\frac{k+1}{2^k}+\left\Vert\left( \frac{k+1}{2^k}\times\frac{1}{2^i}+\vert\lambda_{\varphi(k)}^{(i)}\vert \frac{1}{2^{m_k-1}}\right)_{i\geq 1}\right\Vert_2\\
&\leq 2\vert \lambda_{\varphi(k)}^{(0)}\vert+ \frac{7}{3}\frac{k+1}{2^k}+\left\Vert\left(\lambda_{\varphi(k)}^{(i)}\right)_{i\geq 1}\right\Vert_2 \frac{1}{2^{m_k-1}}.
\end{align*}
Now we use the fact that $\lambda_{\varphi(k)}^{(0)}$ tends to 0 and (\ref{Derargb}) to remark that the preceding expression tends to 0 as $k$ tends to infinity. 
Since the sequence $(y_{k})_{k\in\N}$ is dense in $X$, we conclude that $T$ is $\Gamma_{(\tilde{z_i})_i}^f$-hypercyclic. 
\medskip{}

Let us now deal with the second case that is $\vert\lambda_{n}^{(0)}\vert\rightarrow \infty$, $n\rightarrow \infty$. The proof resembles the previous one, yet we prefer to include all the details. As before we consider the operator $T=B_{w} \bigoplus_{i\geq 1} B_{v}$ defined on the $\ell^2$ direct sum of $\ell^2(\Z)$ spaces $H=\left(\oplus_{i=0}^{\infty}\ell^{2}(\Z)\right)_{\ell^2}$, where the weight $v$ is the same as before (\emph{i.e.} $v_i=2$ if $i>0$; $v_i=1/2$ if $i\leq 0$), but $w$ is now given by
$$w_i=\begin{cases}
1& \text{ if }i>0\\
\frac{1}{2}& \text{ if }i\leq 0
\end{cases}.$$
We also define $(y_{n})_{n\in\N}$ and $(d_n)_{n\in\N}$ as in the previous case and still denote by $F_\frac{1}{v}$ (resp. $F_\frac{1}{w}$) the inverse of $B_v$ (resp. $B_w$), \emph{i.e} the weighted forward shift $F_\nu$ (resp. $F_\omega$) with
$$\nu_i=\begin{cases}
\frac{1}{2}\text{ if } i\geq 0\\
2\text{ if } i<0
\end{cases}\quad \text{(resp. }\omega_i=\begin{cases}
1\text{ if } i\geq 0\\
2\text{ if } i<0
\end{cases}\text{)}.$$
$T$ is not hypercyclic since $\Vert B_w\Vert\leq 1$. Now, as in the previous case, we construct an increasing function $\varphi:\N\to\N$ and an increasing sequence $(m_k)_{k\in\N}\subset\N$ satisfying:
\begin{enumerate}[(i)]
	\item $\Vert\frac{1}{\lambda_{\varphi(k)}^{(i)}} F_{\frac{1}{v}}^{m_k}y_{k,i}\Vert<2^{-k}$ for every $i\geq 1$\label{nbcritIa2};
	\item $\Vert\frac{1}{\lambda_{\varphi(k)}^{(0)}} F_{\frac{1}{w}}^{m_k}y_{k,0}\Vert<2^{-k}$\label{nbcritIb2};
	\item $\Vert\frac{\lambda_{\varphi(j)}^{(i)}}{\lambda_{\varphi(k)}^{(i)}}F_{\frac{1}{v}}^{m_k-m_j}y_{k,i}\Vert<2^{-(k+i)}$ for every $i\geq 1$ and every $j<k$\label{nbcritIIa2};
	\item $\Vert\frac{\lambda_{\varphi(j)}^{(0)}}{\lambda_{\varphi(k)}^{(0)}}F_{\frac{1}{w}}^{m_k-m_j}y_{k,0}\Vert<2^{-k}$ for every $j<k$\label{nbcritIIb2};
	\item $\Vert\frac{\lambda_{\varphi(k)}^{(i)}}{\lambda_{\varphi(j)}^{(i)}}B_{v}^{m_k-m_j}y_{j,i}\Vert<2^{-(k+i)}$ for every $i\geq 1$ and every $j<k$\label{nbcritIIIa2};
	\item $\Vert\frac{\lambda_{\varphi(k)}^{(0)}}{\lambda_{\varphi(j)}^{(0)}}B_{w}^{m_k-m_j}y_{j,0}\Vert<2^{-k}$ for every $j<k$\label{nbcritIIIb2};
	\item $\left(F_{\frac{1}{v}}^{m_k}y_{k,i}\right)(0)=0$ for every $i\geq 1$\label{CritsupIIa2};
	\item $\left(F_{\frac{1}{w}}^{m_k}y_{k,0}\right)(0)=0$;\label{CritsupIIb2}
	\item $\left\vert\lambda_{\varphi(k)}^{(0)}\right\vert\leq \frac{2^{m_k-1}}{k}$;\label{Derarga2}
	\item $\left\Vert\left(\lambda_{\varphi(k)}^{(i)} \right)_i\right\Vert_2\leq \frac{2^{m_k-1}}{k}$.\label{Derargb2}
\end{enumerate}
Once this construction has been made until step $k-1$, since for every $j<k$ and every $m>0$
\[\Vert F_{\frac{1}{w}}^{m-m_j}y_{j,0}\Vert<2^{d_k}\max_{l\leq k}\Vert y_{l,0}\Vert,
\]
then it suffices to choose $\varphi(k)$ such that $\lambda_{\varphi(k)}^{(0)}$ is a large enough element of the converging to $\infty$ sequence $(\lambda_{n}^{(0)})_{n\in\N}$ to ensure that (\ref{nbcritIb2}) and (\ref{nbcritIIb2}) are satisfied. 
Moreover, one can choose $m_k$ sufficiently large for (\ref{Derarga2}) and (\ref{Derargb2}) to hold and also (\ref{nbcritIIIb2}), since for every $y\in c_{00}(\Z)$,
\[B_{w}^{m}y\underset{m\to+\infty}{\longrightarrow}0.\]
Furthermore, thanks to property (\ref{c00y}) of the sequence $(y_n)_{n\in\N}$ we can take $m_k$ even larger in order to satisfy (\ref{nbcritIa2}), (\ref{nbcritIIa2}) and (\ref{nbcritIIIa2}) for every $y\in c_{00}(\Z)$,
\[\ F_{\frac{1}{v}}^{m}y\underset{m\to+\infty}{\longrightarrow}0\text{ and }B_{v}^{m}y\underset{m\to+\infty}{\longrightarrow}0.\]
Finally, up to take $m_k$ bigger again, (\ref{CritsupIIa2}) and  (\ref{CritsupIIb2}) can be satisfied because $d_k\leq k$.  
This finishes the construction of the sequences $(\lambda_{\varphi(k)}^{(i)})_{k\in\N}$ for $i\geq 0$  and $(m_k)_{k\in\N}$. 

Now, thanks to (\ref{nbcritIa2}) and (\ref{nbcritIb2}) we define the orthogonal family $\{\tilde{z_i}\}_{i\in\N}$ in $H$ by
\[\tilde{z_0}:=\left(e_0+\sum_{j=0}^{\infty}\frac{1}{\lambda_{\varphi(j)}^{(0)}} F_{\frac{1}{w}}^{m_j}y_{j,0};0;\ldots\right)\]
and for every $i\geq 1$,
\[\tilde{z_i}:=\left(\underbrace{0;\ldots;0}_{i\text{ times}};e_0+\sum_{j=0}^{\infty}\frac{1}{\lambda_{\varphi(j)}^{(i)}} F_{\frac{1}{v}}^{m_j}y_{j,i};0;\ldots\right).
\]
We deduce from (\ref{nbcritIa2}), (\ref{nbcritIb2}), (\ref{CritsupIIa2}) and (\ref{CritsupIIb2}) that $1\leq\Vert \tilde{z_i}\Vert\leq 1+ \sum_{j=0}^{\infty}2^{-j}= 3$ for every $i\in\N$.  
Now, we normalize this family $z_i:=\frac{\tilde{z_i}}{\Vert\tilde{z_i}\Vert}$ to obtain an orthonormal family $\{z_i\}_{i\in\N}$ in $H$. 
Observe now that the diagonal operator $D:H\to H$ defined by $D\left((x_i)_{i\in\N}\right)=\left(\Vert \tilde{z_i}\Vert^{-1} x_i\right)$ is well-defined since for every $i\in\N$, $\Vert \tilde{z_i}\Vert \geq 0$, has dense range and commutes with $T$. As in the previous case, the properties of $D$ ensure that it suffices to prove that $T$ is $\Gamma_{(\tilde{z_i})_i}^f$-hypercyclic. For $k\in\N$,

\begin{align*}
&\left\Vert T^{m_k}\left(\sum_{i=0}^{\infty}\lambda_{\varphi(k)}^{(i)}\tilde{z_i}\right)-y_{k}\right\Vert^2\\
=&\left\Vert \lambda_{\varphi(k)}^{(0)} B_{w}^{m_k}\left(e_0+\sum_{j=0}^{\infty}\frac{1}{\lambda_{\varphi(j)}^{(0)}} F_{\frac{1}{w}}^{m_j}y_{j,0}\right)-y_{k,0}\right\Vert ^2+\sum_{i=1}^{\infty}\left\Vert \lambda_{\varphi(k)}^{(i)} B_{v}^{m_k}\left(e_0+\sum_{j=0}^{\infty}\frac{1}{\lambda_{\varphi(j)}^{(i)}} F_{\frac{1}{v}}^{m_j}y_{j,i}\right)-y_{k,i}\right\Vert ^2.
\end{align*}

Now using (\ref{nbcritIIa2}) and (\ref{nbcritIIIa2}), we get for every $i\geq 1$,

\begin{align*}
& \left\Vert \lambda_{\varphi(k)}^{(i)} B_{v}^{m_k}\left(e_0+\sum_{j=0}^{\infty}\frac{1}{\lambda_{\varphi(j)}^{(i)}} F_{\frac{1}{v}}^{m_j}y_{j,i}\right)-y_{k,i}\right\Vert \\
\leq & \left\Vert\lambda_{\varphi(k)}^{(i)} B_{v}^{m_k}(e_0)\right\Vert+
\left\Vert \sum_{j<k}\frac{\lambda_{\varphi(k)}^{(i)}}{\lambda_{\varphi(j)}^{(i)}} B_{v}^{m_k}F_{\frac{1}{v}}^{m_j}y_{j,i}\right\Vert+\left\Vert\frac{\lambda_{\varphi(k)}^{(i)}}{\lambda_{\varphi(k)}^{(i)}} B_{v}^{m_k}F_{\frac{1}{v}}^{m_k}y_{k,i}-y_{k,i}\right\Vert+\left\Vert \sum_{j>k}\frac{\lambda_{\varphi(k)}^{(i)}}{\lambda_{\varphi(j)}^{(i)}} B_{v}^{m_k}F_{\frac{1}{v}}^{m_j}y_{j,i}\right\Vert\\
\leq & \left\Vert\lambda_{\varphi(k)}^{(i)} B_{v}^{m_k}(e_0)\right\Vert+\sum_{j<k}\left\Vert\frac{\lambda_{\varphi(k)}^{(i)}}{\lambda_{\varphi(j)}^{(i)}} B_{v}^{m_k-m_j}y_{j,i}\right\Vert+ \sum_{j>k}\left\Vert\frac{\lambda_{\varphi(k)}^{(i)}}{\lambda_{\varphi(j)}^{(i)}} F_{\frac{1}{v}}^{m_j-m_k}y_{j,i}\right\Vert\\
\leq & \left\vert\lambda_{\varphi(k)}^{(i)} \frac{1}{2^{m_k-1}}\right\vert+\sum_{j<k}\frac{1}{2^{k+i}}+\sum_{j>k}\frac{1}{2^{j+i}}\\
\leq & \left\vert\lambda_{\varphi(k)}^{(i)} \frac{1}{2^{m_k-1}}\right\vert+\frac{k+1}{2^k}\times\frac{1}{2^i}.
\end{align*}
On the other hand, similar computations with (\ref{nbcritIIb2}) and (\ref{nbcritIIIb2}) give,

\[\left\Vert \lambda_{\varphi(k)}^{(0)} B_{w}^{m_k}\left(e_0+\sum_{j=0}^{\infty}\frac{1}{\lambda_{\varphi(j)}^{(0)}} F_{\frac{1}{w}}^{m_j}y_{j,0}\right)-y_{k,0}\right\Vert\leq \frac{\vert \lambda_{\varphi(k)}^{(0)}\vert}{2^{m_k-1}}+\frac{k+1}{2^k}.\]
Altogether we get, for every $k\in\N$,

\begin{align*}
\left\Vert T^{m_k}(\sum_{i=0}^{\infty}\lambda_{\varphi(k)}^{(i)}\tilde{z_i})-y_{k}\right\Vert &\leq \sqrt{\left(\frac{\vert \lambda_{\varphi(k)}^{(0)}\vert}{2^{m_k-1}}+\frac{k+1}{2^k}\right)^2+\left\Vert\left( \frac{k+1}{2^k}\times\frac{1}{2^i}+\vert\lambda_{\varphi(k)}^{(i)}\vert \frac{1}{2^{m_k-1}}\right)_{i\geq 1}\right\Vert_2^2}\\
&\leq \frac{\vert \lambda_{\varphi(k)}^{(0)}\vert}{2^{m_k-1}}+\frac{k+1}{2^k}+\left\Vert\left( \frac{k+1}{2^k}\times\frac{1}{2^i}+\vert\lambda_{\varphi(k)}^{(i)}\vert \frac{1}{2^{m_k-1}}\right)_{i\geq 1}\right\Vert_2\\
&\leq \frac{\vert \lambda_{\varphi(k)}^{(0)}\vert}{2^{m_k-1}}+ \frac{7}{3}\frac{k+1}{2^k}+\left\Vert\left(\lambda_{\varphi(k)}^{(i)}\right)_{i\geq 1}\right\Vert_2 \frac{1}{2^{m_k-1}}.
\end{align*}
Now we use (\ref{Derarga2}) and (\ref{Derargb2}) to remark that the preceding expression tends to 0 as $k$ tends to infinity. 
Since the sequence $(y_{k})_{k\in\N}$ is dense in $X$, we conclude that $T$ is $\Gamma_{(\tilde{z_i})_i}^f$-hypercyclic.

\end{proof}

\begin{remark}\label{rem-codim}Given $\Gamma$ not contained in a finite union of vector annuli, the previous proof provides us explicitly with a Hilbert space $H$, a non-hypercyclic operator $T\in \LL(H)$ and an orthonormal family $(z_n)_n$ in $H$ such that $T$ is $\Gamma _{(z_n)_n}$-hypercyclic. But it is quite transparent that the construction imposes that $\overline{\Span}(z_n,\,n\geq 0)$ has \emph{infinite} codimension. This is the technical obstruction which won't allow us to obtain a complete characterization of hypercyclic subsets (see Section \ref{pf-thmA} and the proof of Theorem~A).
\end{remark}

\subsection{Bourdon-Feldman scalar subsets}

Let $X$ be a separable complex Banach space and $T \in \LL(X)$. We recall that the Bourdon-Feldman Theorem \cite{bourdon_somewhere_2003} asserts that any somewhere dense orbit of a single vector $x\in X$ under the action of $T$ is actually dense. In view of this important result, we introduce the following definitions.

\begin{definition}Let $l\geq 1$. We say that $\Gamma \subset \C^l$ (resp. $\Gamma \subset \ell^2(\N)$) is a \emph{Bourdon-Feldman scalar subset} if for every separable Banach space $X$ (resp. separable Hilbert space $H$), for every $T\in \LL(X)$ (resp. $T\in \LL(H)$) and every linearly independent family $(x_1,\ldots ,x_l)$ in $X$ (resp. any orthonormal family $(x_n)_n$ in $H$),
\[\orb(\Gamma _{x_1,\ldots,x_l},T)\text{ somewhere dense in }X \implies \aorb(x_i,T)=X\text{ for some }i\in \{1,\ldots,l\}
\]
(resp.
\[\orb(\Gamma _{(x_n)_n},T)\text{ somewhere dense in }H \implies \aorb(x_i,T)=H\text{ for some }i \geq 0\text{)}.
\]
\end{definition}

The Bourdon-Feldman Theorem says that any non-zero scalar $\lambda \in \C$ is a Bourdon-Feldman scalar subset. This result was improved in \cite{charpentier_-supercyclicity_2016}, where Bourdon-Feldman scalar subsets of $\C$ were characterized. The statement is as follows.

\begin{theorem}\label{thm-BF-cem}A non-empty subset $\Gamma\neq\{0\}$ of $\C$ is a Bourdon-Feldman scalar subset if and only if $\Gamma\T$ is a nowhere dense hypercyclic scalar subset, \emph{i.e.} there exist $0< a\leq b <\infty$ such that $\Gamma \setminus \{0\} \subset [a,b]\T$ and $\text{Int}(\overline{\Gamma\T})=\emptyset$.
\end{theorem}

The only known example of a multidimensional Bourdon-Feldman scalar subset is a finite union of sets of the form $\T f$, $f\in \C^l$ (or $\ell^2(\N)$) non-zero. It can be deduced from \cite[Theorem 3.11]{bayart_dynamics_2009}. Theorems \ref{thm-main} and \ref{thm-main-l2}, together with Theorem \ref{thm-BF-cem}, allow us to obtain a complete characterization of Bourdon-Feldman scalar subsets in $\C^l$ and in $\ell^2(\N)$.

\begin{theorem}\label{thm-BF-scalar}A non-empty subset $\Gamma$ of $\C^l$, $l\geq 1$, (resp. $\ell^2(\N)$) is a Bourdon-Feldman scalar subset if and only if $\Gamma\setminus \{0\}\neq \emptyset$ and there exist $N\in \N$, $g_1,\ldots,g_N$ in $\C^l$ (resp. $\ell^2(\N)$), $\Gamma _1,\ldots \Gamma_N\subset \C$ with $\text{Int}(\overline{\Gamma_i\T})=\emptyset$ for any $1\leq i \leq N$, and $0< a\leq b <\infty$, such that
\[
\Gamma \setminus \{0\} \subset \bigcup _{i=1}^N \Gamma_i g_i.
\]
\end{theorem}

\begin{proof}We first deal with the "if part". Let $\Gamma_1,\ldots ,\Gamma_N$ be as in the theorem. It is enough to check that for any $x_1,\ldots,x_N \in X$, if the orbit of the set $\bigcup _{i=1}^N\Gamma _ix_i$ is somewhere dense in $X$, then $\orb(x_{i_0},T)$ is dense in $X$ for some $i_0\in \{1,\ldots,N\}$. Now, by \cite[Theorem B]{charpentier_-supercyclicity_2016}, if none of the $\orb(x_{i},T)$ is dense in $X$, $1\leq i\leq N$, then each of the orbit $\orb(\Gamma_i x_{i},T)$, $1\leq i\leq N$, is nowhere dense in $X$, hence $\bigcup _{i=1}^N\Gamma _ix_i$ as well.

Let us now turn to the "only if" part. We only give the proof for $\Gamma\subset \ell^2(\N)$, that for $\Gamma \subset \C^l$ being similar and a bit simpler. Let then $\Gamma$ be a Bourdon-Feldman scalar subset in $\ell^2(\N)$. Since $\Gamma$ is in particular a hypercyclic scalar subset, there exist pairwise distinct elements $g_1,\ldots,g_N$ in $\ell^2(\N)$ and $\Gamma _1,\ldots,\Gamma _N$ non-empty subsets of $\C$, with $\Gamma$ bounded and bounded away from $0$ for any $i=1,\ldots,N$, such that
\[\Gamma \setminus \{0\}= \bigcup _{i=1}^N\Gamma _ig_i.\]
Remark first that \cite[Theorem B]{charpentier_-supercyclicity_2016} contains the case $N=1$, thus we may assume that $N\geq 2$. Assume by contradiction that there exist $1\leq i_0\leq N$ so that $\Gamma _{i_0}\T$ is somewhere dense in $\C$ for some $i_0$ and set without loss of generality $i_0=1$.
Recall that we are looking for a Hilbert space $H$, an operator $T$ on $H$ and an orthonormal sequence $(x_n)_n$ in $H$ so that $\orb(\Gamma_{x_1,\ldots,x_N},T)$ is somewhere dense but $T$ is not hypercyclic.

Let $H=\C\oplus\ell^{2}(\Z)$ and $T=e^{\imath\theta}\oplus B_w$, where $e^{\imath\theta}$ is the rotation operator on $\C$ with $\theta\in\R\setminus{\pi\Q}$, and $w$ is a weight defined by \[w_i=\begin{cases}
	2& \text{ if }i>0\\
	\frac{1}{2}& \text{ if }i\leq 0.
\end{cases}\]
$T$ is not hypercyclic, but $B_w$ satisfies the Hypercyclicity Criterion on $\ell^2(\Z)$. Moreover since the rotation $e^{\imath\theta}$ is universal on $\T\subseteq\C$, it is well-known that $e^{\imath\theta}\times B_w$ acting on $\T\times \ell^2(\Z)$ is universal with universal vector $(1,x)$, where $x$ denotes a hypercyclic vector for $B_w$ (see the proof of \cite[Proposition 4.2]{charpentier_-supercyclicity_2016} for example). We can assume that $\Vert x \Vert > \Vert g_1\Vert$. Then, for some $\lambda \in \R_{+}^{*}$, let $F:\ell^2(\N)\to H$ be an isometry such that $F(g_1)=(\lambda, x)$, and set $(x_n)_n:=(F(e_n))_n$.
Observe that $\Gamma_{(x_n)_n}=F(\Gamma)=\Gamma_1 (\lambda,x)\cup_{i=2}^{N}\Gamma_i F(g_i)$ so that
\[\orb(\Gamma_1 (\lambda,x),T)\subseteq\orb(\Gamma_{(x_n)_n},T).\]
Consider the open set $U:= \lambda V\oplus \ell^2(\Z)$ where $V$ is an open set contained in $\overline{\Gamma_1\T}$ and take $(a,y)\in U$.
Then, by definition of $U$ and by compactness of $\T$, there exist a sequence $(\gamma_k)_k$ in $\Gamma_1$ and $\mu\in[0,2\pi[$ such that $\gamma_k\to \frac{\vert a\vert}{\lambda} e^{\imath\mu}$.
Then by universality of $e^{\imath\theta}\times B_w$, there exists a sequence $(n_k)_k$ such that $e^{\imath\theta n_k}\to \frac{a}{\vert a\vert}e^{-\imath\mu}$ and $B_{w}^{n_k}(x)\to y$ as $k\rightarrow \infty$.
This yields 
\[\gamma_k T^{n_k}(\lambda,x)=(\lambda \gamma_k e^{\imath\theta n_k},\gamma_kB_{w}^{n_k}(x))\underset{k\to+\infty}{\longrightarrow} (a,y)\]
and proves that $\orb(\Gamma_1 (\lambda,x),T)$ is somewhere dense in $H$, finishing the proof.
\end{proof}

\section{Applications: Hypercyclic subsets and Bourdon-Feldman subsets}\label{pf-thmA}

We recall that $T\in\LL(X)$ and $\tilde{T}\in\LL(\tilde{X})$ are conjugate to each other if there exists an isomorphism $S:X\rightarrow \tilde{X}$ such that $\tilde{T}\circ S=S\circ T$. The following proposition can be easily checked.
\begin{proposition}\label{conj}Let $X$ be a Banach space, $C$ a non-empty subset of $X$, and $T\in\LL(X)$ and $\tilde{T}\in\LL(\tilde{X})$ be conjugate to each other. Then $T$ is $C$-hypercyclic if and only if $\tilde{T}$ is $S(C)$-hypercyclic, where $S\in \LL(X,\tilde{X})$ is such that $\tilde{T}\circ S=S\circ T$.
\end{proposition}

The next proposition gives an equivalent, but apparently stronger, definition of a hypercyclic subset in $X$.

\begin{proposition}Let $X$ be a separable Banach space and $C$ a non-empty subset of $X$. $C$ is a hypercyclic subset (in the sense of Definition \ref{def-hyp-set}) if and only if for any $T\in \LL(X)$,
$$T\text{ is hypercyclic iff }T\text{ is conjugate to a }C\text{-hypercyclic operator}.$$
\end{proposition}

\begin{proof}
Suppose first that $C$ is a hypercyclic subset.
Then if $T$ is hypercyclic, the argument is the same as in the proof of Proposition \ref{prop-hyp-Gammahyp}. Let $x$ be a hypercyclic vector for $X$ and $z\in C$. We consider a topological isomorphism $S$ of $X$ which maps $z$ to $x$ and define $\tilde{T}:=S^{-1}\circ T\circ S$. By Proposition \ref{conj} again, $\tilde{T}$ is hypercyclic with $z=S^{-1}(x)$ as a hypercyclic vector. Since $z\in C$, $\tilde{T}$ is also $C$-hypercyclic.
For the other way round, we need only use that hypercyclicity is preserved by conjugacy.

For the sufficiency, we need only remark that any operator is conjugate to itself.  
\end{proof}

We now re-state and prove Theorem~A.

\begin{theoA}Let $C$ be a subset of a separable Hilbert space $H$.
\begin{enumerate}
\item We assume that $C$ is contained in a finite dimensional subspace of $H$. Then $C$ is a hypercyclic subset if and only if $C\setminus \{0\}$ is non-empty and contained in a finite union of vector annuli.
\item If $C$ contains a sequence $(x_n)_n$ of linearly independent vectors satisfying
\begin{equation}\label{eq-codim-main-thm1}
\text{Codim}(\overline{\text{Span}}(x_n,\, n\geq 0))=\infty,
\end{equation}
then $C$ is not a hypercyclic subset.
\end{enumerate}
\end{theoA}

\begin{proof}(1) The if part is just Theorem \ref{thm2} and then has already been proven. The proof of the only if part is quite simple now and will be done by \emph{conjugacy}. We start by fixing a separable Hilbert space $H$ and a finite dimensional subset $C$ in $H$ which is not contained in a finite union of vector annuli.
Moreover, we suppose without loss of generality that $0\notin C$. We denote by $l$ the dimension of $\Span(C)$, choose a basis $(x_1,\ldots,x_l)$ of $\Span(C)$, and we denote by $F$ the isomorphism from $\Span(C)$ onto $\C^l$ which maps $x_i$ to $e_i$, $1\leq i\leq l$ (here again $e_1,\ldots,e_l)$ stands for the canonical basis of $\C^l$). We now define $\Gamma =F(C)\subset \C^l$ and observe that obviously $C=\Gamma _{x_1,\ldots,x_l}$.
Then $\Gamma$ is not included in a finite union of vector annuli in $\C^l$ (we already used such implication before the statement of Theorem \ref{thm2}) so that by Theorem \ref{thm-main}, there exist a separable Hilbert space $\tilde{H}$, an operator $\tilde{T}\in\LL(\tilde{H})$ and a linearly independent family $(z_1,\ldots,z_l)$ in $\tilde{H}$ such that $\tilde{T}$ is $\Gamma_{z_1,\ldots,z_l}$-supercyclic but not hypercyclic.
Since $(x_1,\ldots,x_l)$ and $(z_1,\ldots,z_l)$ are finite, there is a topological isomorphism $S$ from $\tilde{H}$ onto $H$ which maps $z_i$ to $x_i$ for every $1\leq i\leq l$. In particular we have
$$C=\Gamma _{x_1,\ldots,x_l}=S(\Gamma _{z_1,\ldots,z_l}).$$
To finish we define $T:=S\circ \tilde{T}\circ S^{-1}$ in $\LL({H})$, observe that $T$ and $\tilde{T}$ are conjugate to each other, and use Proposition \ref{conj} to get that $T$ is $C$-hypercyclic but not hypercyclic. This concludes the proof of (1).

(2) We need only prove that an infinite linearly independent sequence $(x_n)_n$ such that $\Codim(\overline{\Span}(x_n,\,n\geq 0))=\infty$ is not a hypercyclic subset. We fix such a sequence $(x_n)_n$ in some separable Hilbert space $H$. By assumption, there exists an orthonormal basis $(f_m)_{m\in \N}$ of $H$ such that $(f_{2m})_{m\in \N}$ is an orthonormal basis of $\overline{\Span}((x_n,\,n\geq 0))$. We then denote by $F$ the isomorphism from $\overline{\Span}((x_n,\,n\geq 0))$ onto $\ell^2(\N)$ which maps $f_{2m}$ to $e_m$, $m\in \N$ (here again $(e_m)_m$ stands for the canonical basis of $\ell^2(\N)$). We now define $\Gamma =F((x_n)_n)\subset \ell^2(\N)$ and observe that obviously $(x_n)_n=\Gamma _{(f_{2m})_{m}}$.

Since $(x_n)_n$ is linearly independent in $H$, $\Gamma$ is not contained in a finite union of vector annuli and, by Theorem \ref{thm-main-l2}, $\Gamma$ is not a hypercyclic scalar subset. So there exist a separable Hilbert space $\tilde{H}$, a non-hypercyclic operator $\tilde{T}\in \LL(H)$ and an orthonormal sequence $(\tilde{z_m})_{m\in \N} \subset \tilde{H}$ such that
\[
\aorb(\Gamma _{(\tilde{z_m})_m},T)=\tilde{H}.
\]
As mentioned in Remark \ref{rem-codim},
\[\Codim(\overline{\Span}(\tilde{z_m},\,m\geq 0)))=\Codim(\overline{\Span}(f_{2m},\,m\geq 0)))=\infty.
\]
Thus there exists a topological isomorphism $S$ from $\tilde{H}$ onto $H$ which maps $\tilde{z_m}$ to $f_{2m}$, $m\in \N$. In particular we have
\[(x_n)_n=\Gamma _{(f_{2m})_{m}}=S(\Gamma _{(\tilde{z_m})_{m}}).
\]
Finally we define $T:=S\circ \tilde{T}\circ S^{-1}$ in $\LL(H)$, observe that $T$ and $\tilde{T}$ are conjugate to each other, and use Proposition \ref{conj} to infer that $T$ is $(x_n)_n$-hypercyclic but not hypercyclic. This concludes the proof of (2).
\end{proof}

\begin{remark}The previous proof does not allow to extend Theorem~A to any given separable Banach space. The reason is simply that two given separable Hilbert spaces are always topologically isomorphic, what may not be the case for arbitrary separable Banach spaces.
\end{remark}

In the Hilbert setting, the only case that Theorem~A does not cover is that of those linearly independent sequences whose any subsequence spans a subspace with \emph{finite codimensional} closure. This kind of pathological objects are known as \emph{almost overcomplete} sequences.

\begin{definition}
A sequence in a Banach space is called \emph{overcomplete} (resp. \emph{almost overcomplete}) if the closed linear span of each of its subsequences has codimension $0$ (resp. has finite codimension).
\end{definition}

Such sequences have been for instance studied by Klee \cite{Klee1958}, and more recently by Fonf and Zanco, see \cite{Fonf2016,Fonf2014} and the references therein; see also \cite{Baranov2017,Chalendar2006}. We mention that Klee proved that every separable Banach space contains an overcomplete sequence. The second part of Theorem~A can be equivalently restated in terms of almost overcomplete sequences as follows.

\begin{corollary}\label{thmA-aocomplete}Let $C$ be a subset of a separable Hilbert space $H$. If $C$ contains a sequence $(x_n)_n$ of linearly independent vectors which is not an almost overcomplete sequence, then $C$ is not a hypercyclic subset.
\end{corollary}

The first part of Theorem~A provides with various new examples of natural sets which are not hypercyclic subsets.

\begin{corollary}\label{coro-ex-thmA}Let $H$ be a separable Hilbert space. The following subsets of $H$ are not hypercyclic subsets.
\begin{itemize}\item A segment containing $0$;
\item A segment joining two linearly independent vectors;
\item More generally, any sets containing a finite dimensional continuous curve joining two linearly independent vectors;
\item Open sets or spheres with positive radius, as examples of sets of the previous type.
\end{itemize}
\end{corollary}

\begin{remark}As far as we now, even the first example in Corollary \ref{coro-ex-thmA} is new. Indeed, \cite[Theorem A]{charpentier_-supercyclicity_2016} tells that a segment of the form $[a,b]x$ is a hypercyclic subset if $0\notin [a,b]$ and $x\neq 0$, and that $[a,b]$ is not a hypercyclic \emph{scalar} subset if $0\in [a,b]$, but we do not know a reference stating that given \emph{any} Hilbert space $H$, \emph{any} $x\in H$ non-zero, the set $[a,b]x$ is not a hypercyclic subset whenever it contains $0$.
\end{remark}

The second part of Theorem~A also gives rather nice examples of non hypercyclic subsets which are not covered by the first part. Among them, the most natural ones are probably infinite orthonormal families.

\begin{corollary}In a separable Hilbert space, an infinite orthonormal family is never a hypercyclic subset.
\end{corollary}

This corollary is a satisfying answer to Feldman's question about countable hypercyclicity (see Question \ref{Quest4} in the introduction). We shall mention that the counterexample given in \cite[Exercise 6.3.3]{grosse-erdmann_linear_2011} is not an orthonormal family. Yet, in fact, Theorem~A gives a positive answer to Question \ref{Quest5}, that is a \emph{completely} positive answer to Question \ref{Quest4}. The reason is that almost overcomplete bounded sequences enjoy a very strong property, as shown by the following.

\begin{theorem}[Theorem 2.1 of \cite{Fonf2016}; see also Theorem 3.2 of \cite{Chalendar2006}]Each almost overcomplete bounded sequence in a separable Banach space is relatively norm-compact.
\end{theorem}

This theorem implies that a bounded infinite separated sequence in a Banach space cannot be almost overcomplete. Moreover, by compactness, it cannot be contained in any finite dimensional subspace of $H$. Thus it needs to contain a linearly independent sequence which is not an almost overcomplete sequence. Therefore, by Corollary \ref{thmA-aocomplete}, we get:

\begin{corollary}A bounded separated sequence in a separable Hilbert space is a hypercyclic subset if and only if it is finite and not reduced to $\{0\}$.
\end{corollary}

\medskip{}
Now, a natural question is whether any hypercyclic subset contains hypercyclic vectors. Theorem~A together with Theorem \ref{thm2} gives an answer in the finite dimensional case.
\begin{corollary}Let $C$ be a subset of a separable Hilbert space $H$. We assume that $C$ is finite dimensional. If $C$ is a hypercyclic subset, then any $C$-hypercyclic operator admits a hypercyclic vector in $C$.
\end{corollary}

\begin{proof}Since $C$ is a hypercyclic subset, Theorem~A gives $N\geq 1$, $0<a\leq b<\infty$ and $x_1,\ldots,x_N$ in $H$ such that
$$C\subset \bigcup_{i=1}^N[a,b]\T x_i.$$
We can assume that
\begin{equation}\label{eq-coro-2}C\not\subset \bigcup_{\substack{i=1\\i\neq j}}^N[a,b]\T x_i
\end{equation}
for any $1\leq j\leq N$. By Theorem \ref{thm2}, $\lambda x_i$ is hypercyclic for $T$ for some $1\leq i \leq N$ and any $\lambda \neq 0$. Now by \eqref{eq-coro-2} there is some $\lambda \neq 0$ such that $\lambda x_i \in C$.
\end{proof}

\medskip{}
We finish this paragraph by telling that, as in the proof of Theorem~A, and without any extra difficulties, a conjugacy argument can be combined to Theorem \ref{thm-BF-scalar} in order to get a characterization of finite dimensional Bourdon-Feldman subsets. Moreover, since Bourdon-Feldman subsets are also hypercyclic subsets, the second point is clear by Theorem~A. We leave the details to the reader and re-state Theorem~B.

\begin{theoB}Let $C$ be a subset of a separable Hilbert space $H$.
	\begin{enumerate}
		\item If $C$ is finite dimensional, $C$ is a Bourdon-Feldman subset if and only if $C\setminus\{0\}$ is non-empty and if there exists $x_1,\ldots,x_N$ in $X$ and $\Gamma _1,\ldots,\Gamma _N$ subsets of $\C$, with $\Gamma _i\setminus \{0\}$ bounded and bounded away from $0$ and $\Gamma _i\T$ nowhere dense in $\C$ for every $i\in\{1,\ldots,N\}$, such that
		\[
		C\subset \bigcup _{i=1}^N\Gamma _ix_i.
		\]
		\item If $C$ is infinite dimensional and if $C$ contains a sequence $(x_n)_n$ of linearly independent elements satisfying \[\Codim(\overline{\Span}(x_n,\,n\geq 0))=\infty,\]
		then $C$ is not a Bourdon-Feldman subset.
	\end{enumerate}
\end{theoB}

This theorem provides with examples of hypercyclic subsets which are not Bourdon-Feldman subsets (for \emph{e.g.} the sets of the form $[a,b]\T x$ with $0<a < b < \infty$ and $x\in H\setminus \{0\}$).

\section{Open questions}

Regarding to Theorem~A, an answer to the following question would provide with a complete description of hypercyclic subsets in separable Hilbert spaces.

\begin{question}Are almost overcomplete sequences hypercyclic subsets?
\end{question}

The authors think that this question has a negative answer but the technical obstruction mentioned in Remark \ref{rem-codim} makes our construction probably inefficient.

\medskip{}
A description of hypercyclic subsets in Banach or Fr\'echet spaces still remains unknown.

\begin{question}\label{question-general}
Is there a characterization of hypercyclic subsets in some/any Banach or Fr\'echet spaces?
\end{question}

We recall that Feldman \cite{feldman_countably_2003} proved that if $T$ is a unilateral weighted shift then $T$ is hypercyclic if and only if there exists a bounded set having dense $T$-orbit. This is a partial answer to Question \ref{question-general} for $T$ in some specific class.

\medskip{}
Related to $\Gamma$-supercyclicity with $\Gamma\subset \C^l$ is the notion of $n$-supercyclicity introduced by Feldman \cite{feldman_$n$-supercyclic_2002} in 2002. We recall that an operator $T\in \LL(X)$ is said to be $n$-supercyclic, $n\geq 1$, if there exists an $n$-dimensional subspace $E$ of $X$ such that $T$ is $E$-hypercyclic. In particular any $\Gamma$-supercyclic operator, $\Gamma \subset \C^l$, is $l$-supercyclic. So a general problem is the following.

\begin{question}\label{question-l-super}
Given $n,l$ two positive integers, is it possible to describe those subsets $\Gamma\subset \C^l$ for which an operator is $n$-supercyclic if and only if it is $\Gamma$-supercyclic?
\end{question}

Since Feldman proved in \cite{feldman_$n$-supercyclic_2002} that $n$-supercyclicity and $(n-1)$-supercyclicity differ, the above problem need only be considered for $n\leq l$. Actually this question has already been attacked for $n=l=1$, in which situation no complete solution has been formulated, see \cite{charpentier_-supercyclicity_2016} and the references therein.

\medskip{}
In connection with Theorem B, the following question makes also sense.

\begin{question}Is it possible to characterize those subsets $C$ of a separable Banach space $X$ such that the somewhere density in $X$ of $\orb(C,T)$ implies its everywhere density?
\end{question}

\medskip{}
The interactions between linear dynamics and ergodic theory is of great interest in linear dynamics. This is represented by the notions of frequent hypercyclicity and $\mathcal{U}$-frequent hypercyclicity introduced by Bayart and Grivaux \cite{bayart_frequently_2006} and Shkarin \cite{shkarin_spectrum_2009}, respectively.

\begin{question}What can be said about frequent or $\mathcal{U}$-frequent hypercyclicity?
\end{question}

\section*{Acknowledgements}The authors gratefully acknowledge the referee for his careful reading of the manuscript. We would like to thank Richard Aron, Anton Baranov, Manuel Maestre and Vassili Nestoridis for helpful discussions and suggestions. The first author would also like to thank the Laboratoire de Math\'ematiques Pures et Appliqu\'ees de l'Universit\'e du Littoral C\^ote d'Opale for its kind invitation and its financial support. This project was partly supported by the grant ANR-17-CE40-0021 of the French National Research Agency ANR (project Front).

\bibliographystyle{plain}
\bibliography{biblio-jabref}

\end{document}